\documentclass{siamltex}
\usepackage{graphics,stfloats,amssymb,amsmath,amsfonts,epsfig}
\usepackage{algorithm}
\usepackage{algorithmic}
\usepackage{url}
\usepackage{color}

\newcommand{\R}{\mathbb{R}}

\newcommand{\cS}{{\cal S}}
\newcommand{\cE}{{\cal E}}
\newcommand{\cT}{{\cal T}}
\newcommand{\trace}{\mbox{\rm trace}}

\newcommand{\laura}{}

\def\noprint#1{}

\def\swresolved#1{}

\def\bi{\begin{itemize}}
\def\ei{\end{itemize}}
\def\beq{\begin{equation}}
\def\eeq{\end{equation}}
\def\eqnok#1{(\ref{#1})}

\title{Local Convergence of an Algorithm for Subspace Identification from Partial Data}

\author{Laura Balzano\thanks{Department of Electrical Engineering and  Computer Science, University of Michigan. \texttt{girasole@umich.edu}}%
\and Stephen J. Wright\thanks{Department of Computer Sciences, University of Wisconsin-Madison. \texttt{swright@cs.wisc.edu}}}

\begin{document}
%
\maketitle
\begin{abstract}
GROUSE (Grassmannian Rank-One Update Subspace Estimation) is an
iterative algorithm for identifying a linear subspace of $\R^n$ from
data consisting of partial observations of random vectors from that
subspace. This paper examines local convergence properties of GROUSE,
under assumptions on the randomness of the observed vectors, the
randomness of the subset of elements observed at each iteration, and
incoherence of the subspace with the coordinate
directions. Convergence at an expected linear rate is demonstrated
under certain assumptions. The case in which the full random vector is
revealed at each iteration allows for much simpler analysis, and is
also described.  GROUSE is related to incremental SVD methods and to
gradient projection algorithms in optimization.
\end{abstract}

\nocite{Gross09b}
\nocite{RechtImprovedMC09}
\nocite{SteS90}

\section{Introduction} \label{sec:intro}

We seek to identify an unknown subspace $\cS$ of dimension $d$ in
$\R^n$, described by an $n \times d$ matrix $\bar{U}$ whose
orthonormal columns span $\cS$. Our data consist of a sequence of
vectors $v_t$ of the form
\beq \label{eq:vt}
v_t = \bar{U} s_t,
\eeq
where $s_t \in \R^d$ is a random vector whose elements are independent
and identically distributed (i.i.d.) in $\mathcal{N}(0,1)$.
Critically, we observe only a subset $\Omega_t \subset
\{1,2,\dots,n\}$ of the components of $v_t$.

GROUSE \cite{Bal12,BalRN10b} (Grassmannian Rank-One Update Subspace
Estimation) is an algorithm that generates a sequence $\{ U_t
\}_{t=0,1,\dotsc}$ of $n \times d$ matrices with orthonormal columns
with the goal that $R(U_t) \to \cS$ (where $R(\cdot)$ denotes
range). Partial observation of the vector $v_t$ is used to update
$U_t$ to $U_{t+1}$. We present GROUSE (slightly modified from earlier
descriptions) as Algorithm~\ref{grouse:partial}.

\subsection{Applications of Subspace Identification}

Subspace identification problems arise in a great variety of
applications. They are the simplest form of the more general class of
problems in which we seek to identify a low-dimensional manifold in a
high-dimensional ambient space from a sequence of incomplete
observations.  Subspace identification finds applications in
medical~\cite{Ardekani99} and hyperspectral~\cite{Manolakis02}
imaging, communications~\cite{Tong98multichannel}, source localization
and target tracking in radar and sonar~\cite{KrimViberg}, computer
vision for object tracking~\cite{Costeira98}, and in control for
system identification~\cite{vandenberghe12,sysid1996}, where one is
interested in estimating the range space of the observability matrix
of a system. Subspaces have also been used to represent images of a
single scene under varying illuminations~\cite{Basri03} and to model
origin-destination flows in a computer network~\cite{lakhina04}.
Environmental monitoring of soil and crop conditions~\cite{gupchup07},
water contamination~\cite{papadimitriou05}, and seismological
activity~\cite{Wagner96} can all be summarized efficiently by
low-dimensional subspace representations.

\subsection{GROUSE} \label{sec:grouse}

Each iteration of the GROUSE algorithm
(Algorithm~\ref{grouse:partial}) essentially performs a gradient
projection step onto the Grassmannian manifold of subspaces of
dimension $d$, based on the latest partially observed sample
$[v_t]_{\Omega_t}$ of the random vector $v_t \in \cS$. In this
description, we use $[U]_{\Omega_t}$ to denote the row submatrix of the
$n \times d$ matrix $U$ corresponding to the index set
$\Omega_t \subset \{1,2,\dotsc,n\}$. Similarly, $[z]_{\Omega_t}$ denotes
the subvector of elements of $z \in \R^n$ corresponding to elements
of $\Omega_t$. We use $| \Omega_t|$ to denote cardinality of the set
$\Omega_t$ and $\Omega_t^c$ to denote the complement $\{ 1,2,\dotsc,n\}
\setminus \Omega_t$.

\begin{algorithm}
\caption{GROUSE: Partial Data} \label{grouse:partial}
\begin{algorithmic}
\STATE{Given $U_0$, an $n \times d$ matrix with orthonormal columns, with $0<d<n$;}
\STATE{Set $t:=1$;}
\REPEAT 
\STATE{Draw a random subset $\Omega_t \subset \{1,2,\dotsc,n\}$ and
 and observe  $[v_t]_{\Omega_t}$ where $v_t \in \cS$;}
\IF{the eigenvalues of $[U_t]_{\Omega_t}^T[U_t]_{\Omega_t}$ lie 
in the range $[0.5|\Omega_t|/n, 1.5|\Omega_t|/n]$}
\STATE{Define $w_t := \arg \min_w \|[U_t]_{\Omega_t} w - [v_t]_{\Omega_t} \|_2^2$;}
\STATE{Define $p_t := U_t w_t$; 
$[r_t]_{\Omega_t} := [v_t]_{\Omega_t}-[p_t]_{\Omega_t}$; 
$[r_t]_{\Omega_t^c} := 0$;
$\sigma_t:= \|r_t\| \, \|p_t\|$;}
\STATE{Choose $\eta_t>0$ and set}
\begin{equation} \label{eq:gpupdate}
U_{t+1} := U_t + \left[ \left(\cos (\sigma_t \eta_t)-1 \right) \frac{p_t}{\|p_t\|} + \sin(\sigma_t \eta_t) \frac{r_t}{\|r_t\|} \right] 
\frac{w_t^T}{\|w_t\|}.
\end{equation}
\ENDIF
\STATE{$t:=t+1$;}
\UNTIL{termination}
\end{algorithmic}
\end{algorithm}

This description in Algorithm~\ref{grouse:partial} differs from that of~\cite{BalRN10b} only in that the following
condition is required for the eigenvalues of
$[U_t]_{\Omega_t}^T[U_t]_{\Omega_t}$:
\beq \label{step:checkCTC}
\lambda_i ([U_t]_{\Omega_t}^T[U_t]_{\Omega_t}) \in \left[ .5 \frac{|\Omega_t|}{n},
1.5\frac{|\Omega_t|}{n}\right], \quad i=1,2,\dotsc,d,
\eeq
where $\lambda_i(\cdot)$ denotes the $i$th eigenvalue (in decreasing
order). A consequence is that
\beq \label{eq:CTCbound}
\| ([U_t]_{\Omega_t}^T[U_t]_{\Omega_t})^{-1} \| \le \frac{2n}{|\Omega_t|}.
\eeq
As we see later in Theorem~\ref{th:sampledsingvals}, this condition
ensures that the sample $\Omega_t$ is such that $[v_t]_{\Omega_t}$
captures useful information about $\cS$; if it is not satisfied, the
weight vector $w_t$ may not accurately reflect how the latest
observation $[v_t]_{\Omega_t}$ is explained by the current basis
vectors (the columns of $[U_t]_{\Omega_t}$).  Since we need to factor
the matrix $[U_t]_{\Omega_t}$ in order to calculate $w_t$, and since
we have typically that $d \ll n$, the marginal cost of determining or
estimating the singular values of $[U_t]_{\Omega_t}$ and checking the
condition \eqnok{eq:CTCbound} is not excessive. We show in our
analysis that the condition \eqnok{step:checkCTC} is satisfied at most
iterations.

We note several elementary facts about the vector quantities that
appear in GROUSE. Let $P_{R(\cdot)}$ denote the projection operator
onto the range, and $P_{N(\cdot)}$ denote the projection onto the
nullspace of a matrix. Since
\[
[p_t]_{\Omega_t} = P_{R([U_t]_{\Omega_t})}([v_t]_{\Omega_t}), \quad
[r_t]_{\Omega_t} = P_{N([U_t]_{\Omega_t}^T)}  ([v_t]_{\Omega_t}),
\]
we have that
\beq \label{eq:pr}
p_t^T r_t = [p_t]_{\Omega_t}^T [r_t]_{\Omega_t} =0
\eeq
and 
\beq \label{eq:ptr}
\| p_t + r_t\|^2 = \| p_t\|^2 + \| r_t \|^2.
\eeq
By orthonormality of the columns of $U_t$, we  also have that 
\beq \label{eq:pw}
\| p_t \| = \|w_t\|.
\eeq

\subsection{GROUSE in Context} \label{sec:context}

The derivation of GROUSE as a stochastic gradient algorithm on the
Grassmannian manifold can be found in \cite{BalRN10b}, along with a
discussion of its relationship to matrix completion. In this
subsection, we discuss several other aspects of GROUSE's convergence
behavior, focusing on the
regime in which the iterates $U_t$ are close to identifying the
correct subspace $\cS$, so that $\|r_t \| \ll \| p_t \|$. We assume
that the steplength $\eta_t$ is chosen to satisfy
\beq \label{eq:eta.choice.nofudge} 
\sin \sigma_t \eta_t = \frac{\|
  r_t\|}{\|p_t\|}.  
\eeq
Since $1-\cos \sigma_t \eta_t = O(\|r_t\|^2/\|p_t \|^2)$, by
multiplying both sides of \eqnok{eq:gpupdate} by $w_t$, and using
\eqnok{eq:pw}, we have that
\[
U_{t+1} w_t = U_t w_t + \frac{\| r_t\|}{\|p_t\|} \frac{r_t}{\|r_t\|}
\frac{w_t^Tw_t}{\|w_t\|} + O\left( \frac{\|r_t\|^2}{\|p_t \|^2} \right) =
p_t+r_t +  O\left( \frac{\|r_t\|^2}{\|p_t \|^2} \right).
\]
It follows from the definition of $r_t$ that
\begin{subequations} \label{eq:Uw}
\begin{align}
[U_{t+1} w_t]_{\Omega_t} & \approx   [p_t+r_t]_{\Omega_t} = [v_t]_{\Omega_t}, \\
[U_{t+1} w_t]_{\Omega_t^c} & \approx   [p_t+r_t]_{\Omega_t^c} =
[U_t w_t]_{\Omega_t^c},
\end{align}
\end{subequations}
where $\Omega_t^c := \{1,2,\dotsc,n\} \setminus \Omega_t$. Moreover,
in any direction $z$ orthogonal to $w_t$, we have $U_{t+1} z = U_t
z$. Thus, the update \eqnok{eq:gpupdate} has the effect of
(approximately) matching the newly revealed information
$[v_t]_{\Omega_t}$ along the direction $w_t$, while leaving the values
of $U_t w_t$ almost unchanged in the non-revealed components
$\Omega_t^c$, and making no change at all in the remaining
$(d-1)$-dimensional subspace $\{ z \, | \, w_t^Tz=0\}$. In this sense,
\eqnok{eq:gpupdate} is a ``least-change'' update, leaving the current
iterate $U_t$ undisturbed as far as possible, but making just enough
of a change to match the new information. The least-change strategy is
key to the development of quasi-Newton methods for optimization
\cite[Chapter~6]{NocW06}, in which low-rank, least-change updates are
made to approximate Hessian matrices, to match the curvature
information gained in each step.

The relationship of GROUSE to gradient projection becomes clearer when
we define the following measure of inconsistency between $R(U_t)$ and
$\cS$, based on the information revealed in $[v_t]_{\Omega_t}$:
\[
\cE(U_t) := \min_w \| [U_t]_{\Omega_t} w - [v_t]_{\Omega_t} \|_2^2.
\]
It can be shown that
\[
\frac{d \cE}{d U_t}  = -2r_t w_t^T.
\]
With the choice \eqnok{eq:eta.choice.nofudge} of $\eta_t$, we have
from \eqnok{eq:gpupdate} that
\[
U_{t+1} = U_t + \frac{1}{\|p_t\|^2} r_t w_t^T + O\left(
\frac{\|r_t\|^2}{\|p_t \|^2} \right),
\]
so that the GROUSE step is a step in the negative gradient direction
for $\cE$, projected onto the space of $n \times d$
matrices with orthonormal columns.

GROUSE is related too to an incremental singular value decomposition
(ISVD) approach that maintains an approximation $U_t$ with orthonormal
columns, and iterates  in the following way. First, the new random
vector $v_t$ is appended to $U_t$ to form an $n \times (d+1)$ matrix,
with missing elements of $v_t$ ``imputed'' from the current estimate
$U_t$ and the weight vector $w_t$ obtained as in GROUSE. Second, the
SVD of this expanded matrix is computed, and the first $d$ columns of
its left factor (an $n \times (d+1)$ matrix with orthonormal columns)
are taken as the new iterate $U_{t+1}$. (The final column is
discarded.) It is shown in \cite{grouseisvd} that for a certain choice
of steplength parameter $\eta_t$ in GROUSE, the ISVD and GROUSE
algorithms are equivalent.

In our analysis below, we use the following generalization of
\eqnok{eq:eta.choice.nofudge} for the choice of $\eta_t$:
\beq \label{eq:eta.choice}
\sin \sigma_t \eta_t = \alpha_t \frac{\| r_t\|}{\|p_t\|},
\eeq
where $\alpha_t \in (0,2)$ is a user-defined ``fudge factor.'' We show
that the best asymptotic results are obtained by setting $\alpha_t
\equiv 1$.

\subsection{Summary of Results} \label{sec:discussion}

Our main result is expected local linear convergence of the sequence
of subspaces $\{ R(U_t)\}$ to $\cS$.  This section outlines the
assumptions needed to prove our result and discusses their relevance
to computational experience.

We recall the assumption that the observation vector $v_t$ has the
form $\bar{U} s_t$ \eqnok{eq:vt}, with the elements of $s_t$ being
i.i.d. normal with zero mean and identical variance. We assume too
that the set $\Omega_t$ of observed elements of $v_t$ is chosen
independently at each iteration.

The discrepancy between the $d$-dimensional subspaces $R(U_t)$ and
$\cS$ is measured in terms of the $d$ {\em principal angles} between
these subspaces, $\phi_i (U_t,\bar{U})$
\cite[Chapter~5]{SteS90}, which are defined by
\beq \label{eq:defphi} 
\cos \phi_i
(U_t,\bar{U}) = \sigma_i (\bar{U}^TU_t), \;\; i=1,2,\dotsc,d, 
\eeq
where $\sigma_i (\bar{U}^TU_t)$, $i=1,2,\dotsc,d$ are the singular
values of $U_t^T\bar{U}$. The quantity $\epsilon_t$ defined by
\beq \label{eq:epst}
\epsilon_t := \sum_{i=1}^d \sin^2 \phi_i(\bar{U},U_t) =
 \sum_{i=1}^d (1-\sigma_i^2(\bar{U}^TU_t)) =
d- \| \bar{U}^TU_t \|_F^2
\eeq
is central to our analysis. 
%
%
%
%
We show that for small $\epsilon_t$, we have
\beq \label{eq:impr}
\epsilon_{t+1}  \approx \epsilon_t - \frac{\|r_t\|^2}{\|w_t \|^2},
\eeq
and that the expected value of the decrease $\|r_t\|^2/\|w_t\|^2$ is
bounded below by a small multiple of $\epsilon_t$, provided that the
eigenvalue check \eqnok{step:checkCTC} is satisfied. (Higher-order
terms complicate the analysis considerably.)

A critical assumption, made precise below, is {\em incoherence} of the
subspace $\cS$ with respect to the coordinate directions. Concepts of
incoherence have been well studied in the context of compressed
sensing (see for example \cite{CanR07a}).  If $\cS$ were to align
closely with one or two principal axes, then observation subsets
$\Omega_t$ that did not include the corresponding index would be
missing important information about $\cS$. We would need to choose
larger sample sets $\Omega_t$ (of size $|\Omega_t|$ related to $n$), or
to take many more iterations, in order to have a good chance of
capturing the components of $v_t$ that align with $\cS$. 

Our analysis requires another kind of incoherence too. We assume that
the error in $U_t$ revealed by the observation vector --- the part of
$v_t$ that is {\em not} explained by the current iterate $U_t$ --- is
usually incoherent with respect to the coordinate directions. (Our
computations indicate that such is the case.)  This incoherence
measure is denoted by $\mu(x_t)$, where $x_t := (I-U_t U_t^T) v_t$,
and our assumption on this quantity is spelled out in
Lemma~\ref{lem:paren}.

High-probability results play a key role in the analysis. Our lower
bound on the quantity $\|r_t\|^2/\|w_t\|^2$ in \eqnok{eq:impr}, for
instance, is not proved to hold at every iteration but only at a
substantial majority of iterations. In fact, it is possible that
$\epsilon_{t+1} > \epsilon_t$ for some $t$; the sequence $\{
\epsilon_t \}$ may not decrease monotonically.

{\em We state at the outset that the expected linear convergence
  behavior is proved to hold in only a limited regime}, that is, the
main theorem requires $\epsilon_t$ to be quite small and each
$|\Omega_t |$ to be \laura{on the order of $d (\log d) (\log^2 n)$} in
order for the claimed linear rate to be observed. This requirement
on observations is only $\log d$ greater than what is required for
batch matrix completion algorithms~\cite{RechtImprovedMC09}. 
The linear convergence rate observed in computational experiments is,
roughly speaking, a factor of $(1-Xq/(nd))$ per iteration, where $q$
is a lower bound on $|\Omega_t|$ and $X$ is some number not too much
less than one. We see in Section~\ref{sec:computations} that this rate
appears to hold in a much wider regime than the analysis would
strictly predict, \laura{both for much smaller $|\Omega_t|$ and for
  much larger $\epsilon_t$.}  In fact, the same ``gap'' between theory
and practice of local convergence is seen in many optimization
algorithms.
We point out too that the mismatch largely disappears in the full-data
case, where $\Omega_t = \{1,2,\dotsc,n\}$ for all $t$.
In this case, the theoretical restrictions on $\epsilon_t$ are mild,
incoherence is irrelevant, and the predicted convergence behavior
matches closely the computational observations.

\subsection{Outline} \label{sec:outline}

Section~\ref{sec:convergence} contains the proof of our claim of
expected linear convergence. This long section is broken into
subsections, with a ``roadmap'' given at the start.
Section~\ref{sec:full} analyzes the full-data case in which $\Omega_t
\equiv \{1,2,\dotsc,n\}$. Many of the complications of the general
case vanish here, but the specialized analysis holds some interest and
convergence still occurs only in an expected sense, because of the
random nature of the observation vectors $v_t$.

\subsection*{Notation}

As noted earlier, we use $N(\cdot)$ to denote the null space (kernel)
of a matrix and $P_{\cT}$ to denote projection onto a subspace $\cT$.

The notation $\|\cdot \|$ (without subscript) on either vector or
matrix indicates $\| \cdot \|_2$. Recall that the Frobenius norm is
related to $\| \cdot \|_2$ by the following inequalities:
\[
\| A \|_2 \le \| A\|_F \le \sqrt{r} \| A\|_2,
\]
where $r$ is the rank of $A$. We note too that the norms $\| \cdot
\|_2$ and $\| \cdot \|_F$ are invariant under orthogonal
transformations of the matrix argument.

We drop the subscripts $t$ frequently during the paper, when it causes
no confusion to do so, and reminding the reader of this practice where
appropriate.

\section{Expected Linear Convergence} \label{sec:convergence}




We develop the local convergence results for GROUSE in this
section. The analysis is surprisingly technical for such a simple
method, so we break the exposition into relatively short
subsections. We give a brief outline of our proof strategy here.

Subsection~\ref{sec:ett} obtains a lower-bounding expression for the
improvement in the measure $\epsilon_t$ \eqnok{eq:epst} made over a
single step. This bound involves three different quantities, and the rest of the 
paper focuses on controlling each of them. Subsection~\ref{sec:UtUbar} shows that the Frobenius-norm
difference between $U_t$ and $\bar{U}$ can be bounded above and below
by multiples of $\epsilon_t$. Subsection~\ref{sec:samp} examines some
consequences of the fact that only a subset $\Omega_t$ of the elements
of $v_t$ is revealed at each iteration.  This subsection introduces an
assumed lower bound $q$ on the cardinality of $\Omega_t$, and obtains
bounds on $\| r_t\|$ and $\|p_t \|$ (and their ratio) in terms of the
norm of the vector $s_t$ from \eqref{eq:vt}.

Subsection~\ref{sec:UpUr} examines a particular term $(\bar{U}^Tp_t)^T
(\bar{U}^T r_t)$ that appears in the lower-bounding expression for
$\epsilon_t-\epsilon_{t+1}$ obtained earlier in
Subsection~\ref{sec:ett}, deriving bounds for this quantity in terms
of $\epsilon_t$, $\|p_t \|$, and $\|r_t \|$. These bounds are used in
Subsection~\ref{sec:epst1} to make the results of
Subsection~\ref{sec:ett} more precise. 

Subsection~\ref{sec:coherence} defines the concept of coherence used
in this paper, and uses a measure concentration result to show that
the eigenvalue condition \eqnok{step:checkCTC} is satisfied on most
iterations. Subsection~\ref{sec:lb} proves a high-probability bound
for the ratio $\| r_t\|^2/\|p_t\|^2$, which is the dominant term in
the error improvement $\epsilon_t - \epsilon_{t+1}$. This bound is
given in terms of the angle $\theta_t$ that is the angle between
$R(U_t)$ and $\cS$ that is revealed by the (full) random observation
vector $v_t$. Subsection~\ref{sec:expsin} shows that the expected
value of $\sin^2 \theta_t$ is $\epsilon_t/d$.  Finally,
Subsection~\ref{sec:explin} puts the pieces together, proving expected
linear convergence rate by combining bounds for the ``good''
iterations with those for the ``anomalous'' iterations, where the
latter category includes those for which the update is skipped because
condition \eqref{step:checkCTC} fails to hold.

\subsection{A Bound for $\epsilon_t-\epsilon_{t+1}$} \label{sec:ett}

In this subsection, we obtain an expression for
$\epsilon_{t+1}-\epsilon_t$, where $\epsilon_t$ is the quantity
defined in \eqnok{eq:epst}. We deal mostly with the case in which a
step is actually taken by the algorithm, that is, condition
\eqref{step:checkCTC} holds. (If such is not the case, we have
trivially that $\epsilon_{t+1}=\epsilon_t$.)  We start by defining the
$d \times d$ orthogonal matrix $W_t$ as
\beq \label{eq:def.Wt}
W_t := \left[ \frac{w_t}{\|w_t\|}  \, | \, Z_t \right],
\eeq
where $Z_t$ is a $d \times (d-1)$ matrix with orthonormal columns
whose columns span $N(w_t^T)$. It is clear that the first column of
$U_t W_t$ is
\[
\frac{U_t w_t}{\| w_t\|} = \frac{p_t}{\| p_t \|}.
\]
Let us now write the update formula \eqnok{eq:gpupdate} as follows
\begin{align} 
\label{eq:gpupdate.y}
U_{t+1} & := U_t +  \left[ \frac{y_t}{\| y_t \|} - 
\frac{p_t}{\| p_t \|} \right] \frac{w_t^T}{\|w_t\|}, \\
\label{eq:def.y}
\mbox{where} \;\; \frac{y_t}{\| y_t \|} & := 
\cos (\sigma_t \eta_t) \frac{p_t}{\|p_t\|} + 
\sin (\sigma_t \eta_t) \frac{r_t}{\|r_t\|}.
\end{align}
By using a trigonometric identity together with \eqnok{eq:pr}, we can
see that the right-hand side of \eqnok{eq:def.y} has unit norm. From
\eqnok{eq:gpupdate.y}, we have
\begin{align*}
U_{t+1} W_t &= U_t W_t + \left[ \frac{y_t}{\| y_t \|} - 
\frac{p_t}{\| p_t \|} \right] \frac{w_t^T}{\|w_t\|} W_t \\
&= U_t W_t + \left[ \frac{y_t}{\| y_t \|} - 
\frac{p_t}{\| p_t \|} \right] 
\left[ \begin{matrix} 1 & 0 & 0 & \ldots & 0 \end{matrix} \right],
\end{align*}
where $y_t$ is defined in \eqnok{eq:def.y}.  Thus, the update has the
effect of replacing the first column $p_t/\|p_t\|$ of $U_t W_t$ by
$y_t/\|y_t\|$, and leaving the other columns unchanged. Recalling that
the Frobenius norm is invariant under orthogonal transformations,
using \eqnok{eq:epst} and \eqnok{eq:def.y}, and dropping the subscript
$t$ freely on scalars and vectors, we obtain
\begin{align}
\nonumber
\epsilon_t & -\epsilon_{t+1} = \| \bar{U}^T U_{t+1} \|_F^2 - \|
\bar{U}^T U_t \|_F^2 \\
\nonumber
& = \| \bar{U}^T U_{t+1} W_t \|_F^2 - \| \bar{U}^T
U_t W_t \|_F^2 \\
\nonumber
&= \left\|\cos(\sigma \eta) \frac{\bar{U}^T p}{\|p\|} + \sin(\sigma
\eta) \frac{\bar{U}^Tr}{\|r\|}\right\|_2^2 - \left\|
\frac{\bar{U}^Tp}{\|p\|}\right\|_2^2  \\
\nonumber
&=(\cos^2(\sigma \eta)-1) \frac{\|\bar{U}^Tp\|^2}{\|p\|^2}+
2\cos(\sigma \eta) \sin(\sigma \eta)
\frac{(\bar{U}^Tp)^T(\bar{U}^Tr)}{\|p\|\|r\|} + \sin^2(\sigma \eta)
\frac{\|\bar{U}^Tr\|^2}{\|r\|^2}\\
\nonumber
\nonumber
&=\sin^2(\sigma \eta) \left(
\frac{\|\bar{U}^Tr\|^2}{\|r\|^2}-\frac{\|\bar{U}^Tp\|^2}{\|p\|^2}\right)
+ \sin(2\sigma\eta)
\frac{(\bar{U}^Tp)^T(\bar{U}^Tr)}{\|p\|\|r\|} \\
\label{eq:Udecrease.2} 
& \ge -\sin^2 (\sigma \eta)  + \sin(2\sigma\eta)
\frac{(\bar{U}^Tp)^T(\bar{U}^Tr)}{\|p\|\|r\|},
\end{align}
where the final inequality follows from $\| \bar{U}^Tp \| \le \|p \|$
(since the columns of $\bar{U}$ are orthonormal) and $\|\bar{U}^Tr\|^2 / \|r\|^2 \ge
0$.  Choosing $\eta_t$ so that \eqnok{eq:eta.choice} is satisfied, we
have from  $\sin^2 (\sigma \eta) \in [0,1]$ and for any scalar $\beta$ that
\[
\beta \sqrt{1-\sin^2 (\sigma \eta)} \ge \beta - |\beta| \sin^2 (\sigma \eta),
\]
 and thus by substituting \eqnok{eq:eta.choice}, we have
\begin{align*}
\sin (2 \sigma \eta) \beta & = 2 \sin (\sigma \eta) \beta \sqrt{1-\sin^2 (\sigma
  \eta)} \\
& \ge  2 \sin (\sigma \eta) \beta - 2  \sin^3 (\sigma \eta) |\beta| = 2 \alpha
\frac{\|r\|}{\|p\|} \beta -  2 \alpha^3 \frac{\|r\|^3}{\|p\|^3} |\beta|.
\end{align*}
By substituting into \eqnok{eq:Udecrease.2}, we obtain
\beq
\label{eq:Udecrease.3}
\epsilon_t-\epsilon_{t+1} \ge -\alpha^2 \frac{\|r\|^2}{\|p\|^2} 
+ 2 \alpha \frac{\|r\|}{\|p\|} \frac{(\bar{U}^Tp)^T(\bar{U}^Tr)}{\|p\|\|r\|} 
- 2 \alpha^3 \frac{\|r\|^2}{\|p\|^2} \frac{|(\bar{U}^Tp)^T(\bar{U}^Tr)|}{\|p\|^2}.
\eeq

We will return to formula \eqnok{eq:Udecrease.3} in
Section~\ref{sec:UpUr}. To preview: we will show that
\beq \label{eq:upur}
(\bar{U}^Tp)^T(\bar{U}^Tr) \approx \|r\|^2, 
\eeq
and that the final term on the right-hand side is higher-order. Thus,
we can deduce that the right-hand side of \eqnok{eq:Udecrease.3} is
approximately
\[
\alpha (2-\alpha) \frac{\|r\|^2}{\|p\|^2},
\]
and hence that the approximate maximal improvement $\epsilon_t -
\epsilon_{t+1}$ is obtained by setting $\alpha =1$, as claimed
earlier.

\subsection{Relating $U_t$ to $\bar{U}$}
\label{sec:UtUbar}

We state here a fundamental result about the relationship between
$U_t$, $\bar{U}$, and the quantity $\epsilon_t$ defined in
\eqnok{eq:epst}. After an orthogonal transformation, the
squared-Frobenius-norm difference between $U_t$ and $\bar{U}$ is of
the same order as $\epsilon_t$.

Recalling the definition \eqnok{eq:defphi} of the principal angles
$\phi_i(U_t,\bar{U})$ between the subspaces spanned by the columns of
$U_t$ and the columns of $\bar{U}$, we define
\beq \label{eq:defSG}
\Sigma_t := \mbox{\rm diag} \, (\sin \phi_i(\bar{U},U_t)), \quad
\Gamma_t := \mbox{\rm diag} \, (\cos \phi_i(\bar{U},U_t)).
\eeq
Recalling \eqnok{eq:epst} and using the definitions \eqnok{eq:defphi}
and \eqnok{eq:defSG}, we have
\begin{subequations} \label{eq:epsbounds}
\begin{align}
\| \Sigma_t \|_F^2 & =  \sum_{i=1}^d \sin^2 \phi_i(\bar{U},U_t) = \epsilon_t, \\
\| \Gamma_t \|_F^2 &= \sum_{i=1}^d \sigma_i(\bar{U}^TU_t) =
\sum_{i=1}^d \cos^2 \phi_i(\bar{U},U_t) = d-\epsilon_t, \\
\label{eq:epsbounds.U}
\| \bar{U} \bar{U}^T - U_t U_t^T \|_F^2 &= 2d - 2\|\bar{U}^TU_t\|_F^2 =2\epsilon_t.
\end{align}
\end{subequations}

We have the following lemma.
\begin{lemma} \label{uTunearorthogonal}
\laura{Let $\epsilon_t$ be as in \eqnok{eq:epst} and suppose} $n \ge 2d$.
Then there is an orthogonal matrix $V_t \in \R^{d\times d}$ such that
\[
\epsilon_t \le \| \bar{U} V_t - U_t \|_F^2 \le 2 \epsilon_t,
\]
and thus $\|\bar{U}^T U_t - V_t\|_F^2 \leq 2\epsilon_t$.
\end{lemma}
\begin{proof}
The proof uses \cite[Theorem~5.2]{SteS90}. There are unitary
matrices $Q_t$, $\bar{Y}$, and $Y_t$ such that
\beq \label{eq:ss}
Q_t\bar{U}\bar{Y} := \bordermatrix{&d\cr
                d& I \cr
               d & 0 \cr
                n-2d & 0 }, 
\quad
Q_t U_t Y_t := \bordermatrix{&d\cr
                d& \Gamma_t \cr
               d & \Sigma_t \cr
                n-2d & 0 },
\eeq
where $\Gamma_t$ and $\Sigma_t$ are as defined in \eqnok{eq:defSG}.
Defining the orthogonal matrix $V_t := \bar{Y}Y_t^T$ we have that
\begin{align*}
\bar{U}V_t &= Q_t^T \left(\begin{matrix}I \\ 0 \\ 0\end{matrix}\right) 
\bar{Y}^T \bar{Y} Y_t^T \\
&= Q_t^T \left(\begin{matrix}I \\ 0 \\ 0\end{matrix}\right)Y_t^T\\
&=Q_t^T \left(\begin{matrix}\Gamma_t \\ \Sigma_t \\ 0\end{matrix}\right)Y_t^T 
+ Q_t^T \left(\begin{matrix}I-\Gamma_t \\ -\Sigma_t \\ 0\end{matrix}\right)Y_t^T \\
&=U_t + Q_t^T \left(\begin{matrix}I-\Gamma_t \\ -\Sigma_t \\ 0\end{matrix}\right)Y_t^T.
\end{align*}
Therefore, using the abbreviated notation $\phi_i :=
\phi_i(\bar{U},U_t)$, together with orthogonality of $Q_t$ and $Y_t$,
we have
\[
\| \bar{U} V_t - U_t \|_F^2 = \|I-\Gamma_t\|_F^2 + \|\Sigma_t\|_F^2 =
\sum_{i=1}^d [ (1-\cos \phi_i)^2 + \sin^2 \phi_i].
\]
By dropping the cosine part of each summation term, we obtain from
\eqnok{eq:epst} that $\| \bar{U} V_t - U_t \|_F^2 \ge \sum_{i=1}^d
\sin^2 \phi_i = \epsilon_t$, proving the lower bound. For the upper
bound, we have
\begin{align*}
\|U_t^T \bar{U} - V_t^T\|_F^2 
&= \sum_{i=1}^d [ (1-\cos \phi_i)^2 + \sin^2 \phi_i] \\
&= \sum_{i=1}^d [ 2-2 \cos \phi_i]  \le \sum_{i=1}^d [2-2 \cos^2 \phi_i] 
 = 2 \sum_{i=1}^d  \sin^2 \phi_i  = 2 \epsilon_t,
\end{align*}
as required. The final claim is an immediate consequence of this upper
bound.
\end{proof}

\subsection{Consequences of Sampling} \label{sec:samp}

In this subsection we investigate some of the issues raised by
observing the subspace vector $v_t$ only on a sample set $\Omega_t
\subset \{1,2,\dotsc,n\}$, seeing how some of the identities and
bounds of Sections \ref{sec:ett} and \ref{sec:UtUbar} are affected. We state a lower bound
on the cardinality of $\Omega_t$ and an upper bound on $\epsilon_t$
that give sufficient conditions for these looser bounds to hold. These
bounds are vital to the analysis of later subsections.

We start with a simple result about the relationship between
$[\bar{U}]_{\Omega_t}$ and $[U_t]_{\Omega_t}$, based on
Lemma~\ref{uTunearorthogonal}.
\begin{lemma} \label{lem:BCV}
Let $V_t$ be the matrix from Lemma~\ref{uTunearorthogonal}. Then $\|
[\bar{U}]_{\Omega_t} - [U_t]_{\Omega_t} V_t^T \|_F^2 
= \| [\bar{U}]_{\Omega_t} V_t - [U_t]_{\Omega_t} \|_F^2 \le 2 \epsilon_t$.
\end{lemma}
\begin{proof}
We have
\begin{align*}
\| [\bar{U}]_{\Omega_t} V_t - [U_t]_{\Omega_t} \|_F^2
& \le \| \bar{U} V_t - U_t\|_F^2 \le 2 \epsilon_t,
\end{align*}
where the last inequality follows from Lemma~\ref{uTunearorthogonal}.
\end{proof}



We now introduce some simplified notation for important quantities in
our analysis, and state the representations of the key vectors $w_t$,
$p_t$, $w_t$, and $v_t$ in terms of this notation. We also make use of
the vector $s_t$ defined in \eqref{eq:vt}.  As in other parts of the
paper, we drop the subscript $t$ freely on vector quantities.
\begin{subequations}
\begin{align}
\label{eq:def.B}
B & := \bar{U}_{\Omega_t}, \\
\label{eq:def.C}
C & := [U_t]_{\Omega_t}, \\
P_{N(C^T)} & = (I-C(C^TC)^{-1} C^T), \\
P_{N(B^T)} &= (I-B(B^TB)^{-1} B^T), \\
[v_t]_{\Omega_t} &= Bs,  \\
w & = (C^TC)^{-1} C^T Bs, \\
p &= U_t w = U_t (C^TC)^{-1} C^T Bs, \\
\label{eq:pt}
[p_t]_{\Omega_t} &= P_{R(C)} Bs, \\
\label{eq:rtO}
[r_t]_{\Omega_t} &= Bs - [p_t]_{\Omega_t} = (I-C(C^TC)^{-1} C^T) Bs = P_{N(C^T)} Bs.
\end{align}
\end{subequations}
The notation $B$ and $C$ from \eqnok{eq:def.B} and \eqnok{eq:def.C} is
used for simplicity in this subsection and the next. The reader will
note that we have used both $(B^TB)^{-1}$ and $(C^TC)^{-1}$ freely.
This property requires that our assumption \eqnok{step:checkCTC} holds
for $U_t$. It requires a similar property for $[\bar{U}]_{\Omega_t}$,
something we simply assume for now, but prove later as a consequence
of incoherence; see Theorem \ref{th:sampledsingvals}.

Note that we have from 
\eqnok{eq:pt} and \eqnok{eq:rtO} that
\beq \label{eq:dum8}
P_{N(C^T)} [r_t]_{\Omega_t} = P_{N(C^T)} (Bs_t -C w_t) = P_{N(C^T)} Bs_t .
\eeq

For the remainder of the paper, we make the following assumptions on
the size of the sample set $\Omega_t$ and the size of $\epsilon_t$:
\beq \label{eq:Obound} 
| \Omega_t | \ge q, 
\eeq
\beq \label{eq:epsbound.1} 
\epsilon_t \le \frac{1}{128} \frac{q^2}{n^2  d}.  
\eeq
In later subsections, we will derive conditions on $q$ that facilitate
the convergence results. For now, we have the following estimates on
vectors of interest.
\begin{lemma} \label{lem:rps}
Suppose that \eqnok{step:checkCTC} holds and that \eqnok{eq:Obound} and
\eqnok{eq:epsbound.1} are satisfied. Then we have
\begin{align}
\label{eq:rs.bound}
\| r_t\| & \le \sqrt{2 \epsilon_t} \| s_t\|, \\
\label{eq:ps.bound} 
\| p_t \| & \in  \left[ \frac34 \| s_t \|, \frac54 \| s_t \| \right], \\
 \label{eq:rp.5}
\frac{\| r_t \|^2}{\| p_t \|^2} & \le \frac{32}{9} \epsilon_t
\end{align}
\end{lemma}
\begin{proof}
We have from \eqnok{eq:rtO}, $[r_t]_{\Omega_t^c}=0$, the fact that
$P_{N(C^T)}C =0$, and Lemma~\ref{lem:BCV} that
\begin{align*}
\| r_t \| = \|[r_t]_{\Omega_t} \| = \| P_{N(C^T)} Bs_t \| & = \\
\| P_{N(C^T)} (B-CV^T) & s_t \| \le \| B-CV^T\| \| s_t \| \le \sqrt{2
  \epsilon_t} \| s_t\|,
\end{align*}
proving \eqnok{eq:rs.bound}. 

We prove the lower bound in \eqnok{eq:ps.bound} (the upper bound is
similar). We have from $\|C \| \le \| U_t\| = 1$,
Lemma~\ref{lem:BCV}, \eqnok{eq:CTCbound}, \eqnok{eq:Obound}, and
\eqnok{eq:epsbound.1}, that
\begin{align*}
\| p_t \| = \| w_t \| &= \| (C^TC)^{-1} C^TBs_t \| \\
&=  \| (C^TC)^{-1} C^TBV(V^Ts_t) \| \\
&\ge \| (C^TC)^{-1} C^TC(V^Ts_t) \|  - \| (C^TC)^{-1} C^T(BV-C)(V^Ts_t) \| \\
& \ge \| s_t \| - \| (C^TC)^{-1} \| \| C \| \| BV-C \| \| s_t \| \\
& \ge \| s_t \| - \frac{2n}{| \Omega_t |} \sqrt{2 \epsilon_t}  \| s_t \| \\
& \ge \| s_t \| - \frac{2n}{q} \sqrt{2 \epsilon_t}  \| s_t \| \\ 
& \ge \| s_t\| - \frac{2n}{q} \frac{1}{8} \frac{q}{n \sqrt{d}} \|s_t \| \\
& \ge \| s_t \| - \frac{1}{4} \| s_t \| = \frac34 \|s_t\|.
\end{align*}

The final bound \eqnok{eq:rp.5} follows immediately from the preceding
two results.

\end{proof}

\subsection{Estimating $(\bar{U}^Tp)^T(\bar{U}^Tr)$}
\label{sec:UpUr}

We return now to the key quantity $(\bar{U}^Tp)^T(\bar{U}r)$ that
appears in \eqnok{eq:Udecrease.3}, with the goal of establishing a
precise form of the estimate \eqnok{eq:upur}. Throughout this section,
we assume that the conditions \eqnok{eq:Obound} and
\eqnok{eq:epsbound.1} are satisfied.

We start by noting from \eqnok{eq:dum8} that
\begin{align*}
\bar{U}^Tr &= B^T [r_t]_{\Omega_t} = B^T P_{N(C^T)} Bs_t =  B^T P_{N(C^T)} [r_t]_{\Omega_t} \\
\bar{U}^Tp &= \bar{U}^T U_t w_t,
\end{align*}
and therefore
\beq
\label{eq:UrUp}
(\bar{U}^Tr)^T(\bar{U}^Tp) = [r_t]_{\Omega_t}^T P_{N(C^T)} B \bar{U}^T U_t w.
\eeq
By replacing $\bar{U}^TU_t$ with $V+\bar{U}^TU_t - V$ (where $V=V_t$
is the orthogonal matrix from Lemma~\ref{uTunearorthogonal}) and
manipulating, we obtain
\begin{align}
\nonumber
P_{N(C^T)} B \bar{U}^T U_t w &= P_{N(C^T)} B V w + P_{N(C^T)} B (\bar{U}^T U_t - V) w \\
\nonumber
&= P_{N(C^T)} B V (C^TC)^{-1} C^T B s +  P_{N(C^T)} (BV-C) V^T (\bar{U}^T U_t - V) w \\
\nonumber
&=P_{N(C^T)} (B V-C) (C^TC)^{-1} C^T (C + BV - C) V^T s +\\ 
\nonumber
&\qquad P_{N(C^T)} (BV-C)V^T (\bar{U}^T U_t - V) w \\
\nonumber
&=P_{N(C^T)} (B V-C) (C^TC)^{-1} C^TC V^T s + \\
\nonumber
& \qquad  P_{N(C^T)} (B V-C) (C^TC)^{-1} C^T (BV - C) V^T s + \\
\nonumber
&\qquad P_{N(C^T)} (BV-C)V^T (\bar{U}^T U_t - V) w \\
\nonumber
&= P_{N(C^T)} (B-CV^T) s  + P_{N(C^T)} (B-CV^T) z \\
\label{eq:tfo}
&= [r_t]_{\Omega_t} + P_{N(C^T)} (B-CV^T) z,
\end{align}
where the final equality comes from \eqnok{eq:dum8}, and we define
\[
z:= V(C^TC)^{-1} C^T (BV - C) V^T s + (\bar{U}^T U_t - V)
(C^TC)^{-1} C^T Bs.
\]
From \eqnok{step:checkCTC}, we have
\[
\| C \| \le \sqrt{\frac{3| \Omega_t|}{2n}}.
\]
By using this bound, together with \eqnok{eq:CTCbound}, $\| B \| \le
\| \bar{U} \| = 1$, Lemma~\ref{uTunearorthogonal}, and
\eqnok{eq:Obound}, we obtain
\begin{align*}
\|z\| & \leq \|(C^TC)^{-1}\|_2 \|C\|_2 \|BV-C\|_2 \|s\|_2
 + \|\bar{U}^T U_t - V\|_2 \|(C^TC)^{-1}\|_2 \|C\|_2 \|B\|_2 \|s\|_2\\
&\leq \frac{2n}{|\Omega_t|}  \sqrt{\frac{3|\Omega_t|}{2n}}  \sqrt{2\epsilon_t} \|s\|_2 + \sqrt{2\epsilon_t} \frac{2n}{|\Omega_t|}  \sqrt{\frac{3|\Omega_t|}{2n}}  \|s\|_2 \\
&= 4\sqrt{3} \sqrt{\frac{n}{|\Omega_t|}} \sqrt{\epsilon_t} \|s\|_2 
\le  4\sqrt{3} \sqrt{\frac{n}{q}} \sqrt{\epsilon_t} \|s\|_2.
\end{align*}
By combining this bound with \eqnok{eq:tfo} and \eqnok{eq:UrUp}, we obtain
\begin{align}
\nonumber
(\bar{U}^Tr)^T(\bar{U}^Tp)  & \ge \| r \|^2 - \| [r_t]_{\Omega_t}^T P_{N(C^T)} (B-CV^T) z \| \\
\nonumber
& \ge \| r \|^2 - \| r \| \| B-CV^T\| \| z \| \\
\nonumber
& \ge \| r \|^2 - \| r \| \sqrt{2 \epsilon_t} 4 \sqrt{3} \sqrt{\frac{n}{q}} \sqrt{\epsilon_t} \|s\|_2 \\
\label{eq:upur.3}
& \ge \|r\|^2 - 8 \sqrt{3} \sqrt{\frac{n}{q}} \epsilon_t^{3/2} \|s\|^2,
\end{align}
where for the final inequality we used \eqnok{eq:rs.bound}. We apply a
similar argument to obtain an upper bound on
$(\bar{U}^Tr)^T(\bar{U}^Tp)$, leading to a bound on the
absolute value:
\beq \label{eq:upur.5}
|(\bar{U}^Tr)^T(\bar{U}^Tp)|  \le \|r\|^2 + 8 \sqrt{3} \sqrt{\frac{n}{q}} \epsilon_t^{3/2} \|s\|^2.
\eeq
From the bound \eqnok{eq:rs.bound}, we have $\| s_t \|^2 / \|p_t \|^2
\le (4/3)^2$, and thus from \eqnok{eq:upur.3} we obtain
\beq \label{eq:upur.4}
\frac{(\bar{U}^Tp)^T(\bar{U}^Tr)}{\|p\|^2} \ge
\frac{\|r\|^2}{\|p\|^2} - 8 \sqrt{3} \sqrt{\frac{n}{q}} \epsilon_t^{3/2} \frac{16}{9} =
\frac{\|r\|^2}{\|p\|^2} - \frac{128 \sqrt{3}}{9} \sqrt{\frac{n}{q}} \epsilon_t^{3/2}.
\eeq
Likewise, from \eqnok{eq:upur.5}, we have
\beq \label{eq:upur.6}
\frac{|(\bar{U}^Tp)^T(\bar{U}^Tr)|}{\|p\|^2} \le
\frac{\|r\|^2}{\|p\|^2} + \frac{128 \sqrt{3}}{9} \sqrt{\frac{n}{q}} \epsilon_t^{3/2}.
\eeq
Using \eqnok{eq:rp.5} together with the bound \eqnok{eq:epsbound.1} on
$\epsilon_t$, we obtain from \eqnok{eq:upur.6} that
\begin{align}
\nonumber
\frac{|(\bar{U}^Tp)^T(\bar{U}^Tr)|}{\|p\|^2} & \le 
\frac{32}{9} \epsilon_t + \frac{128 \sqrt{3}}{9} \sqrt{\frac{n}{q}} \frac{1}{8 \sqrt{2}} \frac{q}{n \sqrt{d}} \epsilon_t \\
\nonumber
&= \frac{32}{9} \epsilon_t + \frac{8 \sqrt{6}}{9} \sqrt{\frac{q}{nd}} \epsilon_t \\
\label{eq:upur.8}
& \le \frac{64}{9} \epsilon_t,
\end{align}
where the last inequality follows from $q \le n$ and $d \ge 1$.

\subsection{Bounding $\epsilon_{t+1}$} \label{sec:epst1}

We return now to the inequality \eqnok{eq:Udecrease.3}, using the
bounds from the previous subsection to refine our upper bound on
$\epsilon_{t+1}$.  By rearranging \eqnok{eq:Udecrease.3} and
substituting \eqnok{eq:upur.4}, \eqnok{eq:upur.6}, \eqnok{eq:rp.5},
and \eqnok{eq:upur.8}, we obtain
\begin{align*}
\epsilon_{t+1} & \le \epsilon_t + \alpha^2 \frac{\|r\|^2}{\|p\|^2}  - 2 \alpha \frac{(\bar{U}^Tp)^T(\bar{U}^Tr)}{\|p\|^2} + 2 \alpha^3 \frac{\|r\|^2}{\|p\|^2} \frac{|(\bar{U}^Tp)^T(\bar{U}^Tr)|}{\|p\|^2} \\
& \le \epsilon_t + \alpha^2 \frac{\|r\|^2}{\|p\|^2} 
- 2 \alpha  \frac{\|r\|^2}{\|p\|^2} 
+ \frac{256 \sqrt{3}}{9} \alpha \sqrt{\frac{n}{q}} \epsilon_t^{3/2} + 
2 \alpha^3 \left( \frac{32}{9} \epsilon_t \right)
\left( \frac{64}{9} \epsilon_t \right).
\end{align*}
By bounding $\epsilon_t^{1/2}$ using \eqnok{eq:epsbound.1} in the last
term, we obtain
\begin{align*}
\epsilon_{t+1} & \le \epsilon_t - \alpha (2-\alpha) \frac{\|r\|^2}{\|p\|^2} 
+ \frac{256 \sqrt{3}}{9} \alpha \sqrt{\frac{n}{q}} \epsilon_t^{3/2} +
\alpha^3 \frac{64}{9 \sqrt{2}} \frac{q}{n \sqrt{d}} \epsilon_t^{3/2} \\
& \le  \epsilon_t - \alpha (2-\alpha) \frac{\|r\|^2}{\|p\|^2} 
+ \left( \frac{256 \sqrt{3}}{9} \alpha + \frac{64}{9 \sqrt{2}} \alpha^3 \right)
\sqrt{\frac{n}{q}} \epsilon_t^{3/2},
\end{align*}
where we used $q \le n$ and $d \ge 1$ 
in the second inequality.  It is evident from
this expression that $\alpha_t \equiv 1$ is a good choice for the
steplength ``fudge factor'' in \eqnok{eq:eta.choice}. By fixing
$\alpha_t=1$ and simplifying the numerical constants in the
expression above, we obtain
\beq \label{eq:Udecrease.8}
\epsilon_{t+1} \le \epsilon_t - \frac{\|r\|^2}{\|p\|^2} + 55
\sqrt{\frac{n}{q}} \epsilon_t^{3/2}.
\eeq

We proceed in Subsection~\ref{sec:lb} to develop a high-probability
lower bound on $\| r\|^2 / \|p\|^2$, showing that this term is usually
at least a small positive multiple of $\epsilon_t$. Before doing this,
however, it is necessary to discuss the incoherence assumptions and
their consequences.

\subsection{Incoherence and its Consequences} \label{sec:coherence}

It is essential to our convergence results that the subspace $\cS$ to
be identified is {\em incoherent} with the coordinate directions, that
is, the projection of each coordinate unit vector onto $\cS$
should not be too long. This assumption is needed to ensure that the
partially sampled observation vectors have sufficient expected
information content. We make these concept precise in this subsection.

\begin{definition} \label{def:coh}
Given a matrix $U$ of dimension $n \times d$ with orthonormal columns, defining
the subspace $\cT := R(U)$, the {\em coherence} of $\cT$ is 
\[
\mu(\cT):= \frac{n}{d} \max_{i=1,2,\dotsc,n} \, \|P_{\cT} e_i\|_2^2\;,
\]
where $e_i$ is the $i$th unit vector in $\R^n$.  Note that $1 \leq
\mu(\cT) \leq n/d$. Since $P_{\cT} = UU^T$, we have (with a
slight change of notation)
\[
\mu(U) = \frac{n}{d} \max_{i=1,2,\dotsc,n} \, \| U_{i \cdot} \|_2^2,
\]
where $U_{i \cdot}$ denotes the $i$th row of $U$. As a special case of
this definition, we have for a vector $x \in \R^n$ that 
\[
\mu(x) = n \frac{\| x \|_\infty^2}{\|x \|_2^2}.
\]
\end{definition}

We have the following result that relates the coherence of
$R(U_t)$ to that of $R(\bar{U})$, for small values of $\epsilon_t$.
\begin{lemma} \label{lem:cohUt}
Suppose that $| \Omega_t | \ge q$, as in \eqnok{eq:Obound} and that
\beq \label{eq:epsbound.2}
\epsilon_t \le \frac{d}{16n} \mu(\bar{U}).
\eeq
Then 
\[
\mu(U_t) \le \mu(\bar{U}) + 4 \sqrt{\frac{n}{d}} \epsilon_t^{1/2}
\mu(\bar{U})^{1/2} \le 2 \mu(\bar{U}).
\]
\end{lemma}
\begin{proof}
We have by Lemma~\ref{uTunearorthogonal} that
\[
\| (U_t)_{i \cdot} \|_2 
\le \| \bar{U}_{i \cdot} \|_2 + \| \bar{U}_{i \cdot}  V_t - (U_t)_{i \cdot|} \|_2
\le \| \bar{U}_{i \cdot} \|_2 + \sqrt{2 \epsilon_t}, \;\;  i=1,2,\dotsc,d.
\]
By squaring both sides of this inequality and multiplying by $n/d$, we obtain
\[
\frac{n}{d} \| (U_t)_{i \cdot} \|_2^2 \le
\frac{n}{d}  \| \bar{U}_{i \cdot} \|_2^2 + 2^{3/2} \epsilon_t^{1/2} \frac{n}{d}
\| \bar{U}_{i \cdot} \|_2 + 2 \frac{n}{d} \epsilon_t.
\]
By taking the maxima of both sides over $i=1,2,\dotsc,d$, we have from
Definition~\ref{def:coh} that
\[
\mu (U_t) \le \mu(\bar{U}) + 2^{3/2} \sqrt{\frac{n}{d}} \epsilon_t^{1/2}
\mu(\bar{U})^{1/2} + 2 \frac{n}{d} \epsilon_t.
\]
By substituting the bound  \eqnok{eq:epsbound.2} into this expression,
we obtain
\[
\mu (U_t) \le \mu(\bar{U}) + 2^{3/2} \sqrt{\frac{n}{d}} \frac14 
\sqrt{\frac{d}{n}} \mu(\bar{U}) +
2 \frac{n}{d} \frac{d}{16n} \mu(\bar{U}) \le 2 \mu(\bar{U}),
\]
as required.
\end{proof}

We use coherence to analyze the key condition \eqnok{step:checkCTC}
that is used in the algorithm to check acceptability of the sample
$\Omega_t$. We show in the following result that the singular values
of $U_\Omega^T U_\Omega$ are all approximately $|\Omega|/n$, under an
assumption that relates the size of $\Omega_t$ to the coherence of
$U$. The proof of this result appears in
Appendix~\ref{app:sampledsingvals}.
\begin{theorem} \label{th:sampledsingvals}
Given an $n \times d$ matrix $U$ with orthonormal columns and a parameter
$\delta>0$, let $\Omega \subset \{ 1,2,\dotsc,n\}$ be chosen
\laura{uniformly with replacement at random} such that 
\[
|\Omega| > \frac{8}{3} d \mu(U) \log\left(
\frac{2d}{\delta} \right).
\]
Then with probability at least $1-\delta$, the eigenvalues of
$U_\Omega^TU_\Omega$ lie in an interval
\[
\lambda_i\left(U_\Omega^TU_\Omega \right)\in \left[\frac{|\Omega|}{n}
  (1- \gamma), \frac{|\Omega|}{n} (1+\gamma)\right], \quad
i=1,2,\dotsc,d,
\] 
where 
\beq \label{eq:def.gamma}
\gamma := \sqrt{\frac{8 d \mu(U)}{3|\Omega|} \log\left(
  \frac{2d}{\delta} \right)}.
\eeq
\end{theorem}

We conclude this subsection with the following result, which
quantifies the probability that the condition \eqnok{step:checkCTC} is
satisfied. We provide a specific choice for the lower bound $q$ on
sample size that is excessive for current purposes, but useful in
later subsections.
\begin{corollary} \label{co:skips}
Suppose that $\epsilon_t$ satisfies the bounds
\eqnok{eq:epsbound.1} and \eqnok{eq:epsbound.2}.
 Then condition
\eqnok{step:checkCTC} is satisfied with probability at least
$1-\delta$ if $| \Omega_t | \ge q$ and $q$ satisfies
\beq \label{eq:28.1}
\delta \ge 2d \exp \left( \frac{-3q}{64d \mu(\bar{U})} \right).
\eeq
In particular, \eqnok{eq:28.1} is satisfied provided that the
following conditions hold:
\beq \label{eq:28.2}
\delta = .1, \quad
q \ge C_1 (\log n)^2 d \mu(\bar{U}) \log (20 d), \quad
C_1 \ge \frac{64}{3}.
\eeq
\end{corollary}
\begin{proof}
Note first that \eqnok{eq:28.1} is equivalent to
\beq \label{eq:28.3}
\log \left( \frac{2d}{\delta}\right) \le
\frac{3q}{64 d \mu(\bar{U})}.
\eeq
Since \eqnok{eq:epsbound.2} is assumed to hold, we can apply 
Lemma~\ref{lem:cohUt} to obtain
\[
| \Omega_t| \ge q \ge \frac{64}{3} d \mu(\bar{U}) \log \left( \frac{2d}{\delta}\right)
\ge \frac{32}{3} d \mu(U_t) \log \left( \frac{2d}{\delta}\right),
\]
so that the condition of Theorem~\ref{th:sampledsingvals} is
satisfied, with a factor of $4$ to spare.  By applying this theorem,
we obtain
\beq \label{eq:gamsq}
\gamma^2 = \frac{8d \mu(U_t)}{3 | \Omega_t|} 
\log \left( \frac{2d}{\delta}\right) 
\le \frac{16d \mu(\bar{U})}{3q} 
 \log \left( \frac{2d}{\delta}\right) 
\le \frac14,
\eeq
with the final inequality following from \eqnok{eq:28.1}. Thus
\eqnok{step:checkCTC} is satisfied with probability at least
$1-\delta$.

We now verify that \eqnok{eq:28.2} implies \eqnok{eq:28.1}. From the
inequality for $q$ in \eqnok{eq:28.2}, and the assumption that $C_1
\ge 64/3$, we have
\[
q \ge C_1 (\log n)^2 d \mu(\bar{U}) \log (20d)  \ge
\frac{64}{3} d \mu(\bar{U}) \log (20 d).
\]
We obtain by rearranging this expression that
\[
.1 \ge 2d \exp \left( \frac{-3q}{64 d \mu(\bar{U})} \right),
\]
so that the second condition in \eqnok{eq:28.1} holds for the specific
values of $\delta$ and $q$. Thus, we have shown that the values of
$\delta$ and $q$ in \eqnok{eq:28.2} satisfy \eqnok{eq:28.1}, so that
\eqnok{step:checkCTC} is satisfied with probability at least $.9$ for
these values of $\delta$ and $q$.
\end{proof}

\subsection{A High-Probability Lower Bound on $\| r_t\|^2/\|p_t\|^2$}
\label{sec:lb}


Theorem~\ref{th:sampledsingvals} can be used to derive a
high-probability result for a lower bound on the quantity $\| r_t \|^2
/ \|p_t \|^2$, which is the key part of the the error decrease
expression \eqnok{eq:Udecrease.8} and is therefore critical to our
analysis.

We have the following result, which is the main result of
\cite{BalRN10a} and is proved there.
\begin{lemma} \label{lem:resid}
Let $\delta>0$ be given, and suppose that 
\beq \label{eq:Olb}
|\Omega_t| > \frac{8}{3} d
\mu(U_t) \log\left(\frac{2d}{\delta}\right).
\eeq
 Then with probability at least $1-3\delta$, we have
\beq
\label{eq:incresidlower}
\|[v_t]_{\Omega_t} - [p_t]_{\Omega_t}\|_2^2  \geq \left(\frac{|\Omega_t|(1-\xi_t) - d \mu(U_t)\frac{(1+\beta_t)^2}{1-\gamma_t}}{n}\right) \|v_t - U_t U_t^T v_t \|_2^2  ,
\eeq
where we define $x_t := v_t - U_tU_t^T v_t$, set $\gamma_t$ as in
\eqnok{eq:def.gamma}, and define
\beq \label{eq:def.etabeta}
\xi_t :=
\sqrt{\frac{2\mu(x_t)^2}{|\Omega_t|}  \log\left(\frac{1}{\delta}\right)}, 
\quad 
\beta_t := \sqrt{2  \mu(x_t) \log\left(\frac{1}{\delta}\right)}.
\eeq
\end{lemma}

We focus now on the factor in parentheses in \eqnok{eq:incresidlower},
proposing conditions on $\mu(x_t)$ and $q$ under which it can be
bounded below. The conditions on $\mu(x_t)$ are meant to be
``realistic'' in the sense that this quantity is observed to vary like
$\log n$ in practice, so the upper bounds are designed to be a
(possibly large) multiple of this quantity.

\begin{lemma} \label{lem:paren}
Suppose that $\delta=.1$. Suppose that on some iteration $t$, we have
that $| \Omega_t| \ge q$, where $q$ and $C_1$ satisfy the bounds
\eqnok{eq:28.2}
%
Suppose that $\epsilon_t$ satisfies the bounds \eqnok{eq:epsbound.1}
and \eqnok{eq:epsbound.2} for this value of $q$, and that $\mu(x_t)$
satisfies the following two upper bounds
\begin{subequations} \label{eq:muxt}
\begin{align}
\label{eq:muxt.a}
\mu(x_t) &\le \log n \left[ \frac{.045}{\log 10} C_1 d \mu(\bar{U}) \log (20 d) \right]^{1/2} \\
\label{eq:muxt.b}
\mu(x_t) & \le (\log n)^2 \left[ \frac{.05}{8 \log 10} C_1 \log (20 d) \right],
\end{align}
where $x_t = v_t - U_t U_t^T v_t$ as in  Lemma~\ref{lem:resid}.
\end{subequations}
Then we have 
\beq \label{eq:parenlb}
|\Omega_t|(1-\xi_t) - d \mu(U_t)\frac{(1+\beta_t)^2}{1-\gamma_t} \ge \frac{q}{2},
\eeq
where $\xi_t$, $\beta_t$, and $\gamma_t$ are as defined in
Lemma~\ref{lem:resid}.
\end{lemma}
\begin{proof}
We show first that $\xi_t \le .3$, where $\xi_t$ is defined in
\eqnok{eq:def.etabeta}. Since $| \Omega_t| \ge q$ and $\delta=.1$, we
have
\begin{alignat*}{2}
\xi_t^2 &= \frac{2 \mu(x_t)^2}{| \Omega_t|} \log \frac{1}{\delta} \\
& \le \frac{2 \mu(x_t)^2}{q} \log 10 \\
& \le \frac{1}{q} \left[ 2 (\log n)^2 \frac{(.045)C_1}{\log 10} \log (20 d) d \mu(\bar{U}) \right] \log 10 \quad && \makebox{\rm by \eqnok{eq:muxt.a}} \\
& \le \frac{(.09) q}{q} \quad && \makebox{by \eqnok{eq:28.2}} \\
&= .09,
\end{alignat*}
establishing the claim.

We show next that the last term on the left-hand side of
\eqnok{eq:parenlb} is bounded by $.2 q$. The first step is to verify
that $\gamma_t \le .5$ for $\gamma_t$ defined in \eqnok{eq:def.gamma}.
When $\delta=.1$ and $q$ and $C_1$ satisfy the bounds \eqnok{eq:28.2}
(as we assume here), we have 
\[
\gamma_t^2 = \frac{8d \mu(U_t)}{3 | \Omega_t|} \log \left( \frac{2d}{\delta} \right) 
\le \frac{16d \mu(\bar{U})}{3 | \Omega_t|} \log (20d) 
\le \frac{C_1 d \mu(\bar{U})}{4q} \log (20d) \le \frac14,
\]
where we used Lemma~\ref{lem:cohUt} in the first inequality and
\eqnok{eq:28.2} for the remaining inequalities.

We obtain next a bound on $(1+\beta_t)^2$.  From the definition of
$\beta_t$ and the fact that $\mu(x_t) \ge 1$, we have
\[
\beta_t = \sqrt{2 \mu(x_t) \log \frac{1}{\delta} } = 
\sqrt{2 \mu(x_t) \log 10} \ge 1,
\]
so that 
\begin{alignat*}{2}
(1+\beta_t)^2 & \le (2 \beta_t)^2 &&\\
& \le 8 \mu(x_t) \log 10 && \\
& \le (\log n)^2 (.05) C_1 \log (20 d) && \quad \makebox{by \eqnok{eq:muxt.b}} \\
& \le (.05) \frac{q}{d \mu(\bar{U})} && \quad \makebox{by \eqnok{eq:28.2}.}
\end{alignat*}

By using these last two bounds in conjunction with
Lemma~\ref{lem:cohUt}, we obtain
\[
d \mu(U_t) \frac{(1+\beta_t)^2}{1-\gamma_t} \le 2d \mu(\bar{U}) 
\frac{(.05) \frac{q}{d \mu(\bar{U})}}{.5} \le \frac{.1}{.5} q = .2q. 
\]
The result \eqnok{eq:parenlb} follows by combining this bound with
$1-\xi_t \ge .7$ (proved earlier) and $| \Omega_t | \ge q$ (assumed).
\end{proof}


We now derive a high-probability lower bound on $\| r_t \|^2/\|p_t
\|^2$.
\begin{lemma} \label{lem:rplow}
Suppose that $| \Omega_t | \ge q$ for all $t$, where $q$ and $C_1$
satisfy the bounds \eqnok{eq:28.2}. Suppose that
$\epsilon_t$ satisfies the bounds \eqnok{eq:epsbound.1} and 
\eqnok{eq:epsbound.2}. Suppose that
there is a quantity $\bar{\delta} \in (0,.6)$ such that the bounds
\eqnok{eq:muxt} are satisfied by $x_t = v_t - U_t U_t^T v_t$ with
probability at least $1-\bar{\delta}$. Then with probability at least
$.6-\bar{\delta}$, we have that
\[
\frac{\|r_t\|^2}{\|p_t\|^2} \ge (.32) \frac{q}{n} \sin^2 \theta_t,
\]
where $\theta_t$ is the angle between $v_t$ and $R(U_t)$.
\end{lemma}
\begin{proof}
Since we assume \eqnok{eq:28.2}, and thus that $\delta=.1$, we have 
from Corollary~\ref{co:skips} that the check on the eigenvalues of 
$[U_t]_{\Omega_t}^T[U_t]_{\Omega_t}$ in
\eqnok{step:checkCTC} is satisfied with
probability at least $1-\delta = 1-.1 = .9$. From
Lemma~\ref{lem:resid}, we have that \eqnok{eq:incresidlower} holds
with probability at least $1-3 \delta = 1-.3 = .7$. We have assumed
further that \eqnok{eq:muxt} holds with probability
$1-\bar{\delta}$. Thus, from the union bound, we have under our
assumptions that the bounds \eqnok{step:checkCTC},
\eqnok{eq:incresidlower}, and \eqnok{eq:muxt} all hold with
probability at least $.6-\bar{\delta}$.  Since, in particular, the
conditions \eqnok{step:checkCTC}, \eqnok{eq:Obound}, and
\eqnok{eq:epsbound.1} are satisfied under this scenario, we have
from Lemma~\ref{lem:rps} that 
\[
\| p_t \| \le \frac54 \| s_t \| = \frac54 \| v_t \|.
\]
By using this bound together with the definitions of
$[r_t]_{\Omega_t}$ and $[r_t]_{\Omega_t^c}$ in
Algorithm~\ref{grouse:partial} and Lemmas~\ref{lem:resid} and
\ref{lem:paren}, we obtain
\[
\frac{\|r_t\|^2}{\|p_t\|^2} 
\ge \frac{16}{25} \frac{\|[r_t]_{\Omega_t}\|^2}{\|s_t\|^2}
= \frac{16}{25} \frac{\|[v_t]_{\Omega_t} - [p_t]_{\Omega_t} \|^2}{\|v_t\|^2} 
\ge \frac{16}{25} \frac{q}{2n} 
\frac{\|v_t - U_t U_t^T v_t \|_2^2}{\|v_t\|^2}.
\]
Using orthonormality of the columns of $U_t$ and the definition of $\cos \theta_t$,
we obtain
\[
 \frac{\|v_t - U_t U_t^T v_t \|_2^2}{\|v_t\|^2} =
\frac{\| v_t\|^2 - v_t^TU_t U_t^T v_t}{\|v_t\|^2} =
1-\frac{[v_t^T(U_t U_t^T v_t)]^2}{\|v_t\|^2 \| U_t U_t^T v_t\|^2} =
1-\cos^2 \theta_t = \sin^2 \theta_t.
\]
We complete the proof by combining the last two expressions.
\end{proof}

\subsection{Expectation for the Angle Captured by $v_t$}
\label{sec:expsin}

Here we obtain an expected value for the quantity $\sin^2 \theta_t$,
where $\theta_t$ is the angle between the (full) random sample vector
and the subspace $R(U_t)$. Noting that $v_t = \bar{U} s_t$, where
$s_t$ is random, we have
\beq \label{eq:cos.theta}
\cos^2 \theta_t = \frac{[(\bar{U} s_t)^T (U_t U_t^T \bar{U} s_t)]^2}{\| \bar{U} s_t\|^2 \| U_t U_t^T \bar{U} s_t \|^2} = 
\frac{s_t^T \bar{U}^T U_t U_t^T \bar{U} s_t}{\| s_t \|^2}.
\eeq

We start with two elementary technical results.
\begin{lemma} \label{lem:esi}
Let $w \in \R^d$ be a random vector whose components $w_i$,
$i=1,2,\dotsc,d$ are independent and identically distributed. Then
\beq \label{eq:esi}
E_w \left( \frac{w_i^2}{\sum_{j=1}^d w_j^2} \right) = \frac{1}{d}.
\eeq
\end{lemma}
\begin{proof}
By the additive property of expectation, we have
\[
1 = E \left( \frac{\sum_{i=1}^d w_i^2}{\sum_{j=1}^d w_j^2} \right) 
= \sum_{i=1}^d E \left( \frac{w_i^2}{\sum_{j=1}^d w_j^2} \right) = 
d E \left( \frac{w_i^2}{\sum_{j=1}^d w_j^2} \right), \;\; 
i=1,2,\dotsc,d,
\]
since each of the $w_i$ is identically distributed.
\end{proof}

\begin{lemma} \label{lem:es}
Given any matrix $Q \in \R^{d \times d}$, suppose that $w \in \R^d$ is
a random vector whose components $w_i$, $i=1,2,\dotsc,d$ are all
i.i.d. $\mathcal{N}(0,1)$. Then
\[
E \left( \frac{w^TQw}{w^Tw} \right) =\frac{1}{d} \trace \, Q.
\]
\end{lemma}
\begin{proof}
\begin{align*}
E \left( \frac{w^TQw}{w^Tw} \right) & = 
\sum_{i \neq j} E \left( \frac{w_i w_j Q_{ij}}{\|w\|^2} \right) +
\sum_{i=1}^n E \left( \frac{w_i^2 Q_{ii}}{\|w\|^2} \right) \\
&= \sum_{i=1}^n Q_{ii}  E \left( \frac{w_i^2}{\|w\|^2} \right) = \frac{1}{d} \trace \, Q,
\end{align*}
where the second equality follows from Lemma~\ref{lem:esi} and the
fact that $E(w_iw_j/\|w\|^2) = 0$ for $i \neq j$.
\end{proof}

The main result of this subsection follows.
\begin{lemma} \label{lem:esint}
Suppose that $s_t \in \R^d$ is a random vector whose components are
i.i.d. $\mathcal{N}(0,1)$.  Then
\[
E(\sin^2 \theta_t) = \epsilon_t/d,
\]
where $\epsilon_t$ is defined in \eqnok{eq:epst}.
\end{lemma}
\begin{proof}
From \eqnok{eq:cos.theta}, Lemma~\ref{lem:es}, and \eqnok{eq:epst} we
have
\[
E(\cos^2 \theta_t) 
= \frac{1}{d} \trace (  \bar{U}^T U_t U_t^T \bar{U}  ) 
= \frac{1}{d} \|  U_t^T \bar{U} \|_F^2   
= \frac{1}{d} (d-\epsilon_t)  
= 1-\frac{\epsilon_t}{d},
\]
giving the result.
\end{proof}

\subsection{Expected Linear Decrease} \label{sec:explin}

We now put the pieces of theory derived in the previous subsections
together, to demonstrate the expected decrease in $\epsilon_t$ over a
single iteration. 


%
\begin{theorem} \label{th:erate}
Suppose that $| \Omega_t | \ge q$ for all $t$, where $q$ and $C_1$
satisfy the bounds \eqnok{eq:28.2}. Suppose that $\epsilon_t$
satisfies the bounds \eqnok{eq:epsbound.1} and
\eqnok{eq:epsbound.2}. Suppose that there is a quantity $\bar{\delta}
\in (0,.6)$ such that the bounds \eqnok{eq:muxt} are satisfied by $x_t
= v_t - U_t U_t^T v_t$ with probability at least
$1-\bar{\delta}$. Suppose that at each iteration, $s_t$ in 
\eqnok{eq:vt} is a random
vector whose components are i.i.d. $\mathcal{N}(0,1)$. Then
\beq \label{eq:erate} 
E[\epsilon_{t+1} \, | \, \epsilon_t] \le
\epsilon_t - (.32) (.6-\bar{\delta}) \frac{q}{nd} \epsilon_t + 55
\sqrt{\frac{n}{q}} \epsilon_t^{3/2}.  
\eeq
\end{theorem} 
\begin{proof}
Under the given assumptions, we have from \eqnok{eq:Udecrease.8} and
Lemma~\ref{lem:rplow} that
\[
\epsilon_{t+1} \le \epsilon_t - .32 \frac{q}{n} \sin^2 \theta_t
+ 55 \sqrt{\frac{n}{q}} \epsilon_t^{3/2}, \quad \makebox{with probability at least $.6-\bar{\delta}$,}
\]
while
\[
\epsilon_{t+1} \le \epsilon_t
+ 55 \sqrt{\frac{n}{q}} \epsilon_t^{3/2}, \quad \makebox{otherwise.}
\]
(``Otherwise'' includes iterations on which no step is taken because
condition \eqref{step:checkCTC} fails to hold; we have
$\epsilon_{t+1}=\epsilon_t$ on these iterations.)  We obtain the proof
by combining these two results and using Lemma~\ref{lem:esint}.
\end{proof}

\begin{corollary} \label{co:estep}
Suppose that the conditions of Theorem~\ref{th:erate} hold and that in
addition, $\epsilon_t$ satisfies the following bound:
\beq  \label{eq:epsbound.3}
\epsilon_t \le (8 \times 10^{-6}) (.6-\bar{\delta})^2 \frac{q^3}{n^3 d^2}.
\eeq
We then have
\beq \label{eq:rate}
E[ \epsilon_{t+1} \, | \, \epsilon_t ] \le \left( 1- (.16) (.6-\bar{\delta})
\frac{q}{nd} \right) \epsilon_t.
\eeq
\end{corollary}
\begin{proof}
Given the bound \eqnok{eq:epsbound.3}, we have
\[
55 \sqrt\frac{n}{q} \epsilon_t^{1/2}
\le
55 \sqrt\frac{n}{q} (.0029) (.6-\bar{\delta}) \frac{q^{3/2}}{n^{3/2}d}
\le (.16)  (.6-\bar{\delta}) \frac{q}{nd}.
\]
We obtain the result by combining this bound with \eqnok{eq:erate}.
\end{proof}

This result indicates that the rate of decrease in error metric
$\epsilon_t$ is more rapid for higher values of the sampling ratio
$q/n$, and becomes slower as subspace dimension $d$ increases. The
expected decrease in \eqnok{eq:rate} is consistent with the factor
$(1-1/d)$ that we prove in the next section for the full-data case
($q=n$), modulo the factor $(.16) (.6-\bar{\delta})$. The appearance
of the latter factor is of course due to the uncertainty caused by
sampling.



\section{The  Full-Data Case: $q=n$} \label{sec:full}



When a random full vector $v_t \in \cS$ is available at each iteration
of GROUSE (that is, $\Omega_t \equiv \{1,2,\dotsc,n\}$), the algorithm
and its analysis simplify considerably, as we show in this
section. The expected decrease factor in $\epsilon_t$ at each
iteration is asymptotically $(1-1/d)$.

While the ISVD algorithm is preferred for this no-noise, full-data
case, we note that gradient algorithms are more flexible than
algorithms based explicitly on linear algebra
when additional constraints or regularizers are present, such as a
sparsity regularizer on the data fit or factor weights. It may be
possible to build on the full-data analysis presented in this section
to obtain a convergence guarantee for a Grassmannian gradient-descent
algorithm on such regularized problems.

Algorithm~\ref{grouse:full} shown below is the specialization of
Algorithm~\ref{grouse:partial} to the full-data case. Since
$(U_t)_{\Omega_t}^T (U_t)_{\Omega_t} = U_t^TU_t = I$ for all $i$, the
eigenvalue check \eqnok{step:checkCTC} is no longer needed.
\begin{algorithm}
\caption{GROUSE: Full Data} \label{grouse:full}
\begin{algorithmic}
\STATE{Given $U_0$ and $n \times d$ matrix with orthonormal columns, with $0<d<n$;}
\FOR{$t=0,1,2,\dotsc$} 
\STATE{Take $v_t \in \cS$;}
\STATE{Define $w_t := \arg \min_w \|U_t w - v_t \|_2^2 = U_t^Tv_t$;}
\STATE{Define $p_t := U_t w_t$; $r_t := v_t-p_t$; $\sigma_t:= \|r_t\| \, \|p_t\|$;}
\STATE{Choose $\eta_t>0$ and set}
\begin{equation} \label{eq:gupdate}
U_{t+1} := U_t + \left[ \left(\cos (\sigma_t \eta_t)-1 \right) \frac{p_t}{\|p_t\|} + \sin(\sigma_t \eta_t) \frac{r_t}{\|r_t\|} \right] 
\frac{w_t^T}{\|w_t\|}.
\end{equation}
\ENDFOR
\end{algorithmic}
\end{algorithm}
The definitions of certain quantities above are simplified in the
full-data case, as we demonstrate here (with the introduction of
notation $A_t := U_t^T \bar{U}$):
\begin{subequations} \label{eq:defs}
\begin{align}
v_t &= \bar{U} s_t, \\
\label{eq:defA}
A_t &:= U_t^T \bar{U}, \\
w_t &= U_t^Tv_t = U_t^T \bar{U} s_t = A_t s_t, \\
p_t &= U_t w_t, \\
r_t &= v_t - U_t w_t = (I- U_t U_t^T) \bar{U} s_t.
\end{align}
\end{subequations}
We continue to use $\theta_t$ to denote the angle between $\cS$ and
$R(U_t)$ that is exposed by the update vector $v_t$. We have
\beq \label{eq:rtheta}
\cos \theta_t = \frac{\| w_t \|}{\| v_t \|} = \frac{\|p_t \|}{\|v_t\|}, \quad
\sin \theta_t = \frac{\sqrt{\|v_t\|^2-\|w_t\|^2}}{\|v_t\|} = \frac{\|r_t\|}{\|v_t\|}.
\eeq
Thus
\beq \label{eq:sigma}
\sigma_t = \|r_t \| \| p_t \| = \| v_t\|^2 \sin \theta_t \cos \theta_t
 = \frac12 \| v_t \|^2 \sin 2 \theta_t.
\eeq
By using $A_t$ defined in \eqnok{eq:defA} we have from \eqnok{eq:epst}
that
\beq \label{eq:UUTF.A}
\epsilon_t = d - \| A_t \|_F^2 = d - \trace (A_t A_t^T).
\eeq

Our first result provides an exact expression for the relationship
between $\epsilon_{t+1}$ and $\epsilon_t$.
It also motivates an ``optimal'' choice for $\eta_t$ consistent with
the one discussed in Subsection~\ref{sec:context}.  The proof of this
result is quite technical, involving various trigonometric identities
and elementary linear algebra manipulations, so we relegate it to
Appendix~\ref{app:lemA}.

\begin{lemma} \label{lem:A}
We have for all $t$ that
\beq
\label{eq:Aincr}
\epsilon_t - \epsilon_{t+1}
 = \frac{\sin (\sigma_t \eta_t)\sin(2 \theta_t - \sigma_t \eta_t)}{\sin^2 \theta_t}
  \left( 1- \frac{w_t^T A_t A_t^T w_t}{w_t^Tw_t} \right).
\eeq
Moreover, the right-hand side is nonnegative for $\sigma_t \eta_t \in
(0,2\theta_t)$, and zero if  $v_t \in R(U_t) = \cS_t$
or $v_t \perp \cS_t$ (that is, $\theta_t=0$ or $\theta_t = \pi/2$).
\end{lemma}

The expression \eqnok{eq:Aincr} immediately suggests the following
choice for $\eta_t$:
\beq \label{eq:eta.full}
\eta_t := \frac{\theta_t}{\sigma_t} = \frac{2 \theta_t}{\| v_t \|^2 \sin 2 \theta_t},
\eeq
for which $\sin \sigma_t \eta_t = \| r_t \| / \| v_t \|$. (In the
regime $\|r_t \| \ll \|p_t\|$, this choice is similar to
\eqnok{eq:eta.choice.nofudge} made for the general case, because of
\eqnok{eq:ptr}.)
Given \eqnok{eq:eta.full}, \eqnok{eq:Aincr}
simplifies to
\beq \label{eq:Aincr.eta}
\epsilon_t - \epsilon_{t+1}  =  \left( 1- \frac{w_t^T A_t A_t^T w_t}{w_t^Tw_t} \right)
\eeq

We now proceed with an expected convergence analysis for the choice of
$\eta_t$ in \eqnok{eq:eta.full}, for which the convergence bound is
\eqnok{eq:Aincr.eta}.
\begin{theorem} \label{lem:fullconv}
Suppose that in Algorithm~\ref{grouse:full}, $v_t = \bar{U} s_t$,
where the components of $s_t$ are chosen i.i.d. from
$\mathcal{N}(0,1)$, and that $\eta_t$ chosen as in \eqnok{eq:eta.full}
for each $t$. Suppose too that $\epsilon_t \le \bar{\epsilon}$ for
some $\bar{\epsilon} \in (0,1/3)$. Then 
\beq \label{eq:explin.full} 
E\left[ \epsilon_{t+1} \, | \, \epsilon_t \right] \le \left( 1- \frac{1-3 \epsilon_t}{d} \right) \epsilon_t.  \eeq
\end{theorem}
\begin{proof}
From  Lemma~\ref{uTunearorthogonal}, using the $d
\times d$ orthogonal matrices $Y_t$ and $\bar{Y}$ defined in
\eqnok{eq:ss}, we have that
\[
A_t = U_t^T \bar{U} = Y_t \Gamma_t \bar{Y}^T,
\]
so that 
\[
A_t^T A_t = \bar{Y} \Gamma_t^2 \bar{Y}^T, \quad
A_t^T A_t A_t^T A_t = \bar{Y} \Gamma_t^4 \bar{Y}^T.
\]
Thus for the critical term in \eqnok{eq:Aincr.eta}, using $w_t = U_t^T
\bar{U} s_t = A_t s_t$, dropping the subscript $t$ freely, and
recalling the definition \eqnok{eq:defSG} of $\Gamma_t$, we can write
\beq \label{eq:elf.1}
\frac{w^T A A^T w}{w^Tw} = \frac{s^T A^T A A^T A s}{s^T A^T A s} =
\frac{\tilde{s}^T \Gamma^4 \tilde{s}}{\tilde{s}^T \Gamma^2 \tilde{s}} =
  \frac{\sum_{i=1}^d \tilde{s}_i^2 \cos^4 \phi_i}{\sum_{i=1}^d
    \tilde{s}_i^2 \cos^2 \phi_i},
\eeq
where $\tilde{s}=\bar{Y}^Ts$ and $\phi_i = \phi_i(\bar{U},U_t)$. Since
the components of $s$ are chosen i.i.d. from $\mathcal{N}(0,1)$, the
components of $\tilde{s}$ are also i.i.d. from $\mathcal{N}(0,1)$. 

We make two useful observations before proceeding. First, from the
definition of $\epsilon_t$ in \eqnok{eq:epst}, we have
\beq \label{eq:psibd}
0 \le \frac{\sum_{i=1}^d \tilde{s}_i^2 \sin^2 \phi_i}{\sum_{i=1}^d
  \tilde{s}_i^2} \le \max_{i=1,2,\dotsc,d} \sin^2 \phi_i \le \epsilon_t.
\eeq
Second, for any scalar $u$ with $0 \le u \le \bar\epsilon$, for any
$\bar\epsilon \in [0,1/2)$, we have
\[
\frac{1}{1-u} =  1 + \frac{u}{1-u} \le 1 + 2u.
\]

Returning to \eqnok{eq:elf.1}, dropping the indices on the
summation terms for clarity, introducing the notation
\begin{equation} \label{eq:psi}
\psi_t := \frac{\sum \tilde{s}_i^2 \sin^2 \phi_i}{\sum \tilde{s}_i^2},
\end{equation}
and noting from \eqnok{eq:psibd} that $\psi_t \in [0,\epsilon_t]$, 
we have
\begin{align*}
 \frac{\sum \tilde{s}_i^2 \cos^4 \phi_i}{\sum
    \tilde{s}_i^2 \cos^2 \phi_i} &=
\frac{\sum \tilde{s}_i^2 [1- 2 \sin^2 \phi_i + \sin^4 \phi_i]}{\sum \tilde{s}_i^2 (1-\sin^2 \phi_i)} \\
& \le 
\frac{\sum \tilde{s}_i^2 - (2-\epsilon_t) \sum \tilde{s}_i^2 \sin^2 \phi_i}{\sum \tilde{s}_i^2 -\sum \tilde{s}_i^2 \sin^2 \phi_i}   && 
\mbox{from \eqref{eq:psibd}} \\
&=
\frac{1-  (2-\epsilon_t) 
\psi_t}{1- \psi_t} && \mbox{from \eqref{eq:psi}} \\
& \le
\left[ 1-  (2-\epsilon_t)  \psi_t  \right]
\left[
1 + \frac{1}{1-\epsilon_t} \psi_t
\right] \\
& = 1- \left[ 2 - \epsilon_t - \frac{1}{1-\epsilon_t} \right] \psi_t
- \frac{2-\epsilon_t}{1-\epsilon_t} \psi_t^2 \\
& \le  1- \left[ 2 - \epsilon_t - \frac{1}{1-\epsilon_t} \right] \psi_t \\
& \le 1 - \left[ 2 - \epsilon_t - (1+2\epsilon_t) \right] \psi_t \\
&= 1-(1-3 \epsilon_t) \psi_t.
\end{align*}
We have too from Lemma~\ref{lem:es} that
\[
E(\psi_t) = \frac{1}{d} \sum_{i=1}^d \sin^2 \phi_i = \frac{\epsilon_t}{d},
\]
where the expectation is taken over  $s_t$. 
Assembling these results, we have that
\[
E \left( 1- \frac{w_t^T A_t A_t^T w_t}{w_t^Tw_t} \right)
= E \left( 1-  \frac{\sum \tilde{s}_i^2 \cos^4 \phi_i}{\sum
    \tilde{s}_i^2 \cos^2 \phi_i} \right) 
\ge (1-3 \epsilon_t)  E(\psi_t) =
(1-3 \epsilon_t)  \frac{\epsilon_t}{d}.
\]
The result follows by taking the conditional expectation of both sides
in \eqnok{eq:Aincr.eta}, and using the bound just derived.
\end{proof}

This result shows that the sequence $\{ \epsilon_t \}$ converges
linearly in expectation with an asymptotic rate of $(1-1/d)$. This
rate allows for some interesting observations. First, if $d=1$, it
suggests convergence in a single step --- as indeed we would expect,
as the full vector $v_t \in \cS$ would in this case reveal the 
solution in one step. More generally, we have from the bound
\[
(1-1/d)^d \le \frac{1}{e}
\]
that a decrease factor of about $e$ can be expected over each set of
$d$ consecutive iterations. By comparison, the same amount of
information --- $d$ full vectors randomly drawn from $\cS$ --- is
sufficient to reveal the subspace completely (with probability
$1$). We could obtain an orthonormal basis by assembling these $d$
vectors into a $n \times d$ matrix and performing a singular value
decomposition (SVD). (Of course, extension of an SVD-based approach to
the case of partial data is not straightforward.)


\section{Computational Results} \label{sec:computations}

We present some computational results on random problems to illustrate
the convergence behavior of GROUSE in both the partial-data and
full-data cases. 

For the full-data case, we implemented Algorithm~\ref{grouse:full} in
Matlab on a problem for which the $n \times d$ subspace $\cS$ was
chosen randomly, as the range space of an $n \times d$ matrix whose
elements are i.i.d. in $\mathcal{N}(0,1)$. 
We used a random starting matrix $U_0$ whose columns were orthonormalized. 
Figure~\ref{fig:full} shows results for $n=10000$ and $d=4$,
$d=6$. The straight line in this semilog plot ($t$ vs $\log
\epsilon_t$) represents the predicted asymptotic convergence rate
$(1-1/d)$, while the irregular line represents the actual error. There
is a close correspondence between these results and the predictions
of Theorem~\ref{lem:fullconv}. On early iterations, when $\epsilon_t$
is large, convergence is slower than the asymptotic rate, as predicted
by the presence of the factor $1-3 \epsilon_t$ in the expression
\eqnok{eq:explin.full}. On later iterations, this factor approaches
$1$, and the asymptotic rate emerges --- the curve of actual errors
becomes parallel to the straight line.

\begin{figure} 
\begin{tabular}{cc}
\includegraphics[width=2.2in]{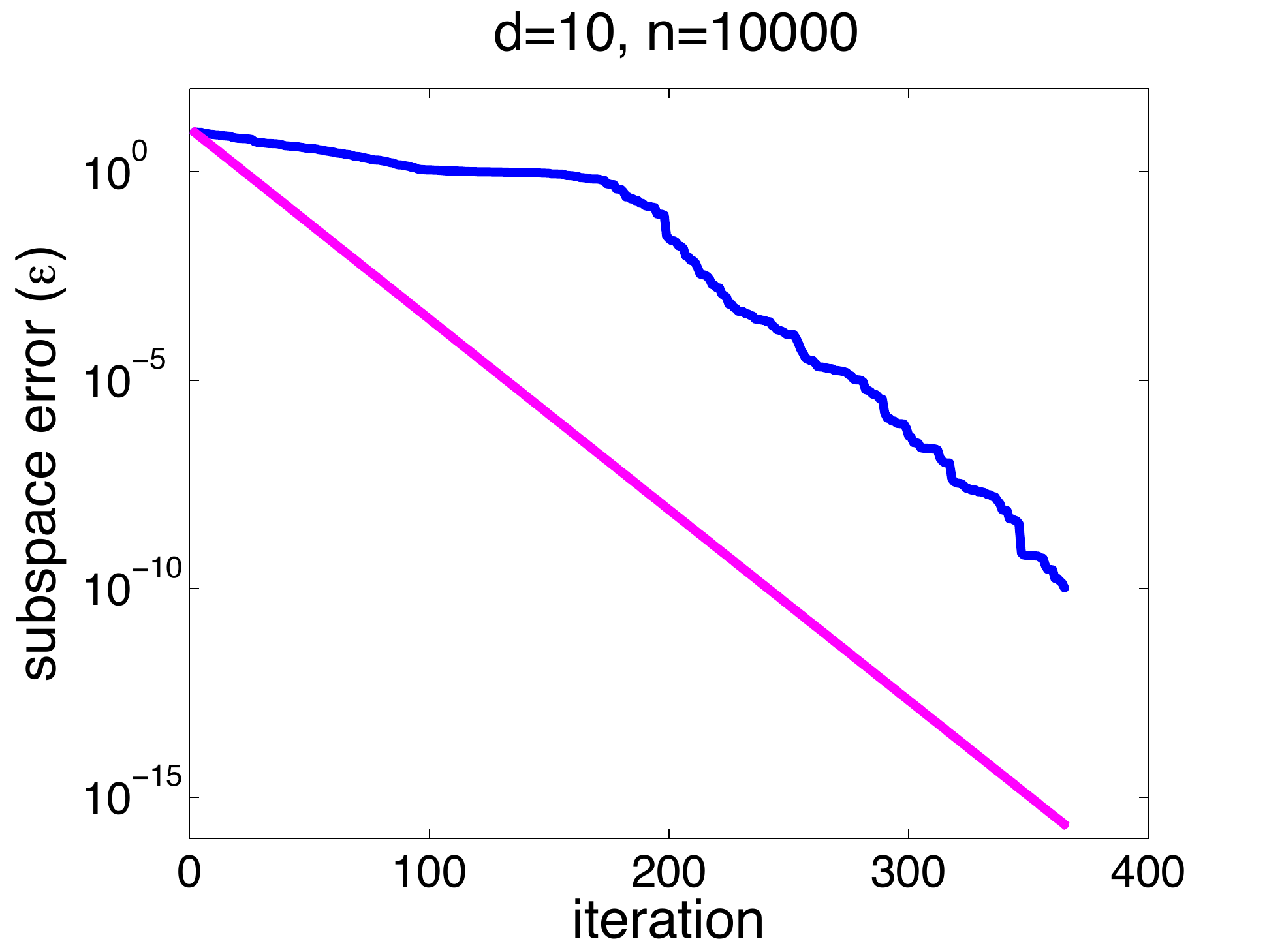} & 
\includegraphics[width=2.2in]{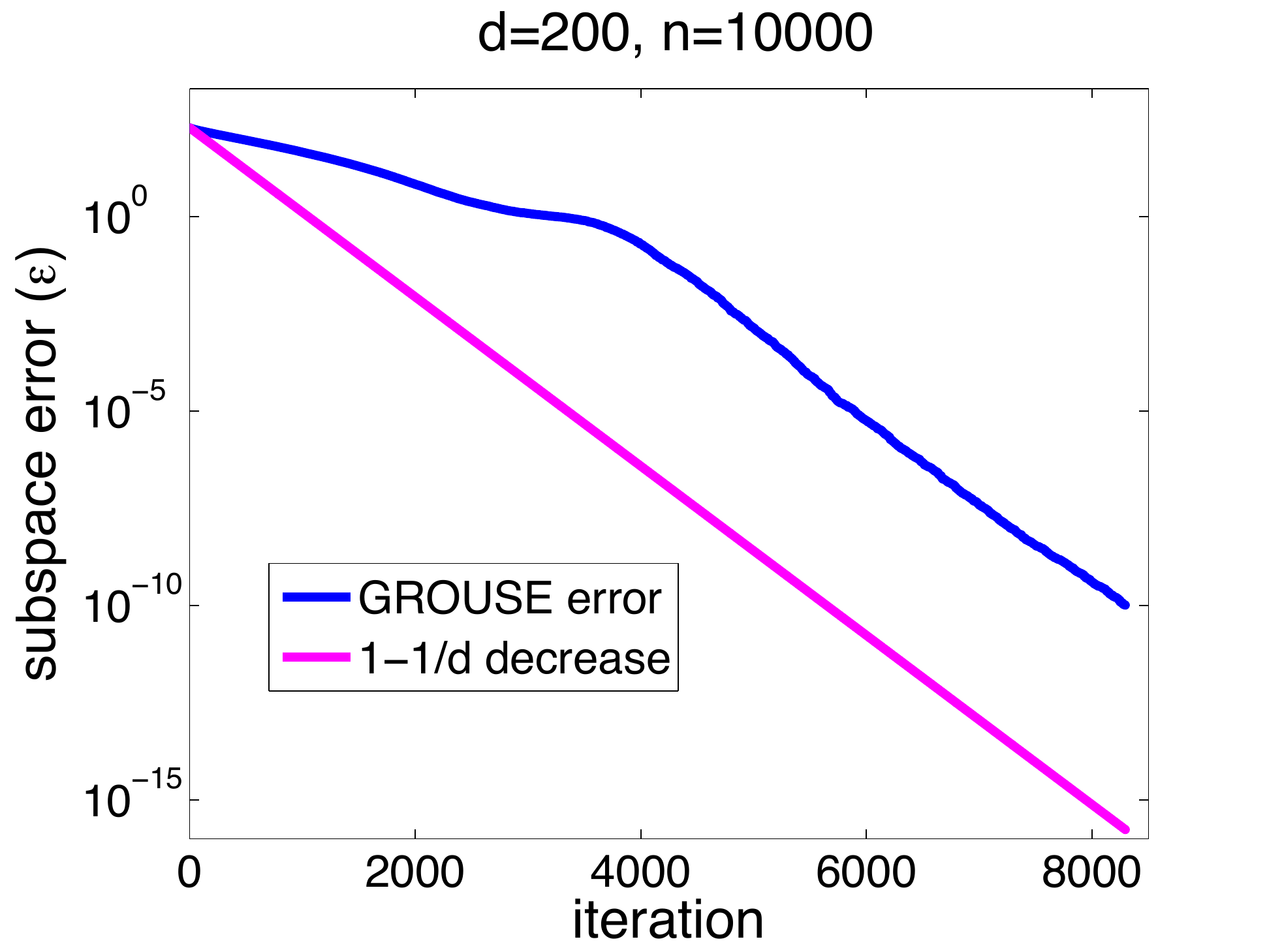} 
\end{tabular}
\caption{Illustration of convergence on Full-Data Case for $n=10000$ with $d=10$ (left)
and $d=200$ (right).\label{fig:full}}
\end{figure}

For the general case, we chose various values of the dimensions $n$
and $d$ and the sampling cardinality $q$, and ran a number of trials
that were constructed in the following manner.  The target space $\cS$
was defined to be the range space of an $n \times d$ matrix $T$ whose
entries were chosen i.i.d. from $\mathcal{N}(0,1)$, and $\bar{U}$ was
obtained by orthonormalizing the columns of $T$. To obtain a starting
matrix $U_0$, we added to $T$ an $n \times d$ matrix whose elements
were chosen i.i.d. from $\mathcal{N}(0,1/4)$, and orthonormalized the
resulting matrix. We then generated vectors from $\bar{U}$ using
Gaussian vectors $s_t$, and updated $U_t$ using the GROUSE
algorithm. In the computational experiments we did not check whether
condition \eqnok{step:checkCTC} was satisfied, and instead took every
step. When $|\Omega|$ was sufficiently large, in alignment with the
theory, this bound was almost always satisfied. As a specific example,
for $n=10,000$ and $|\Omega|=d \log d \log n$, out of 1000 trials, the
bounds of \eqnok{step:checkCTC} were satisfied 98.5\% of the time for
$d=10$ and 100\% of the time for $d=100$.

After running each trial for a large enough number of iterations $N$
to establish an asymptotic convergence rate, we computed the value $X$
to satisfy the following expression:
\beq \label{eq:X} \epsilon_N = \epsilon_0 \left(
1-X\frac{q}{nd} \right)^N, 
\eeq
By comparing with \eqnok{eq:rate}, we see that $X$ absorbs the factor
$(.16) (.6-\bar{\delta})$ that is independent of $n$, $d$, and $q$,
and that arises because of the errors introduced by sampling.  

The values of $X$ for various values of $n$, $d$, and $q$ are shown in
Figure~\ref{fig:phase}. For all $q$ larger than some modest multiple
of $d$, $X$ is not too far from $1$, showing that the actual
convergence rate is not too much different from $(1-q/(nd))$ and that
indeed the analysis is somewhat conservative. 


\begin{figure} 
\begin{tabular}{cc}
\includegraphics[width=2.3in]{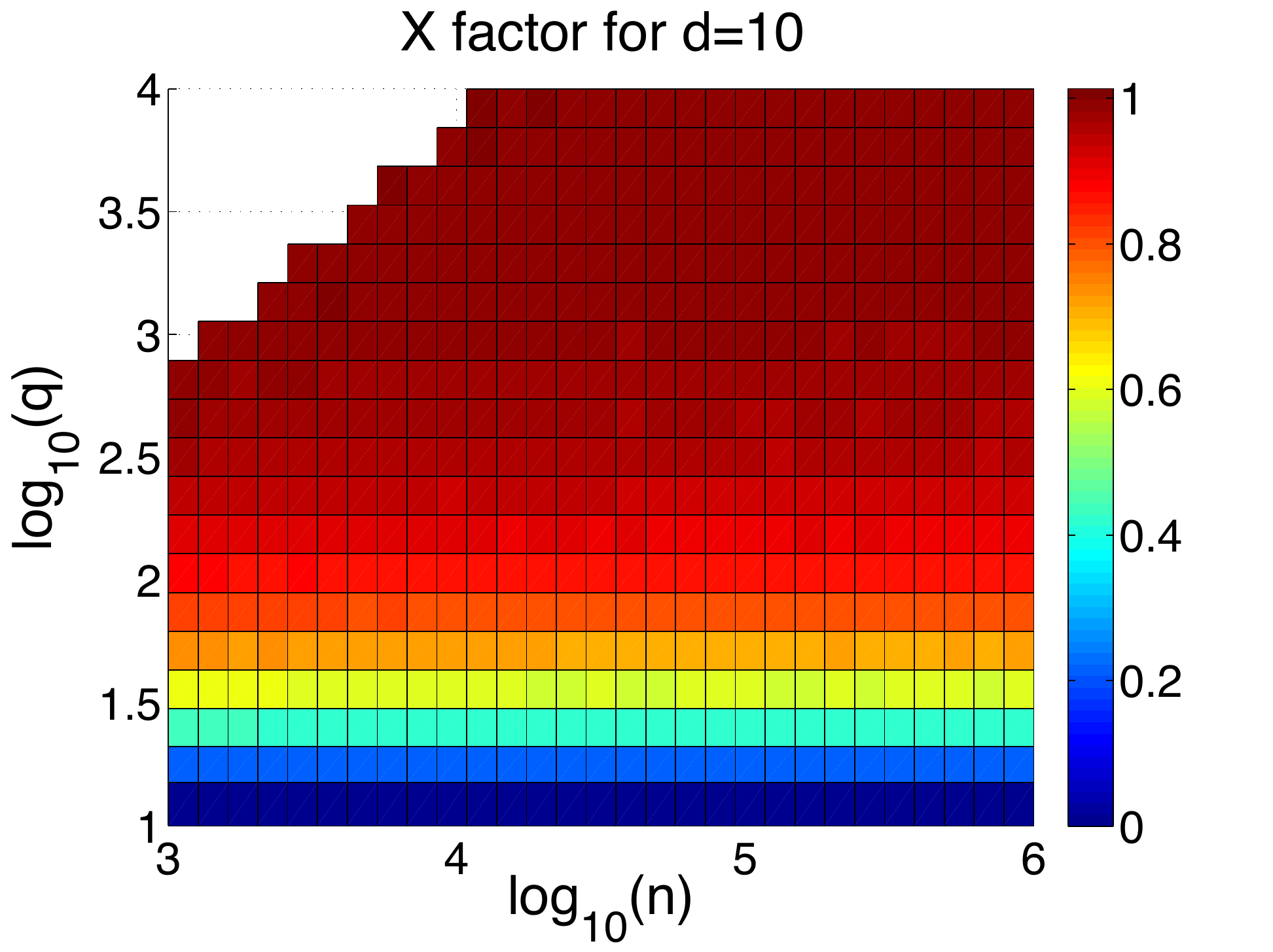} & 
\includegraphics[width=2.3in]{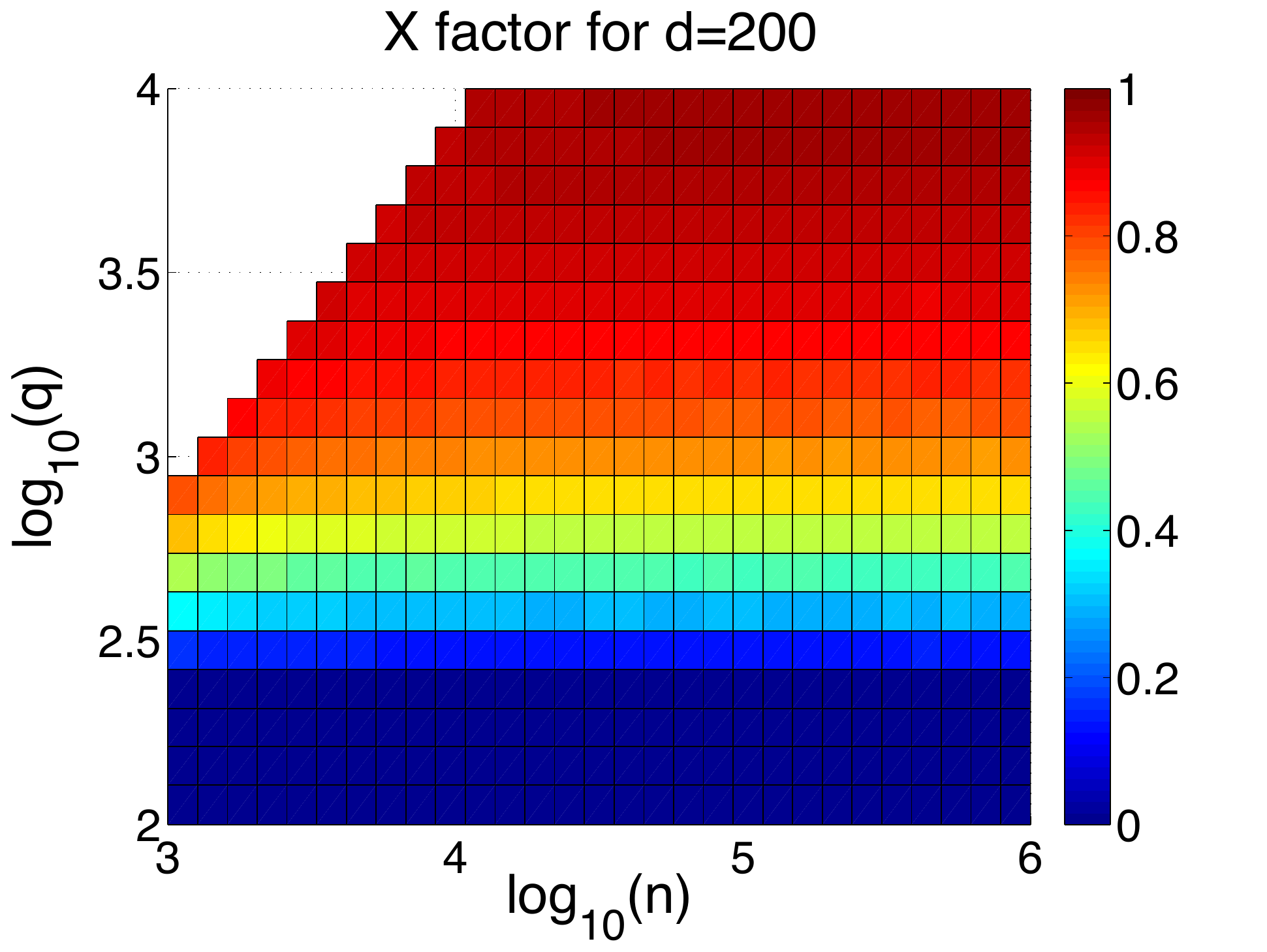}  \\
(a) & (b) \\ \\
\includegraphics[width=2.3in]{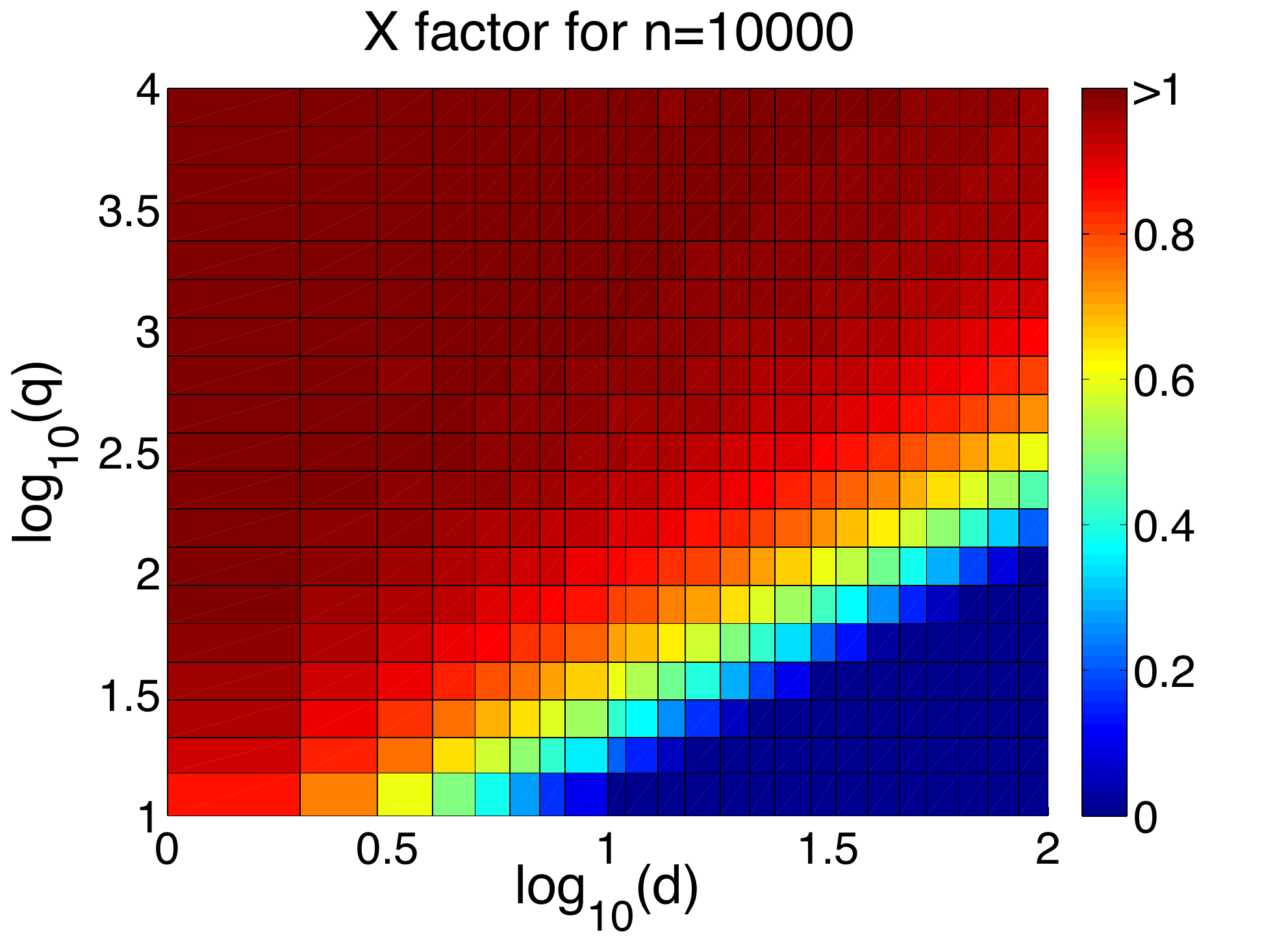} & 
\includegraphics[width=2.3in]{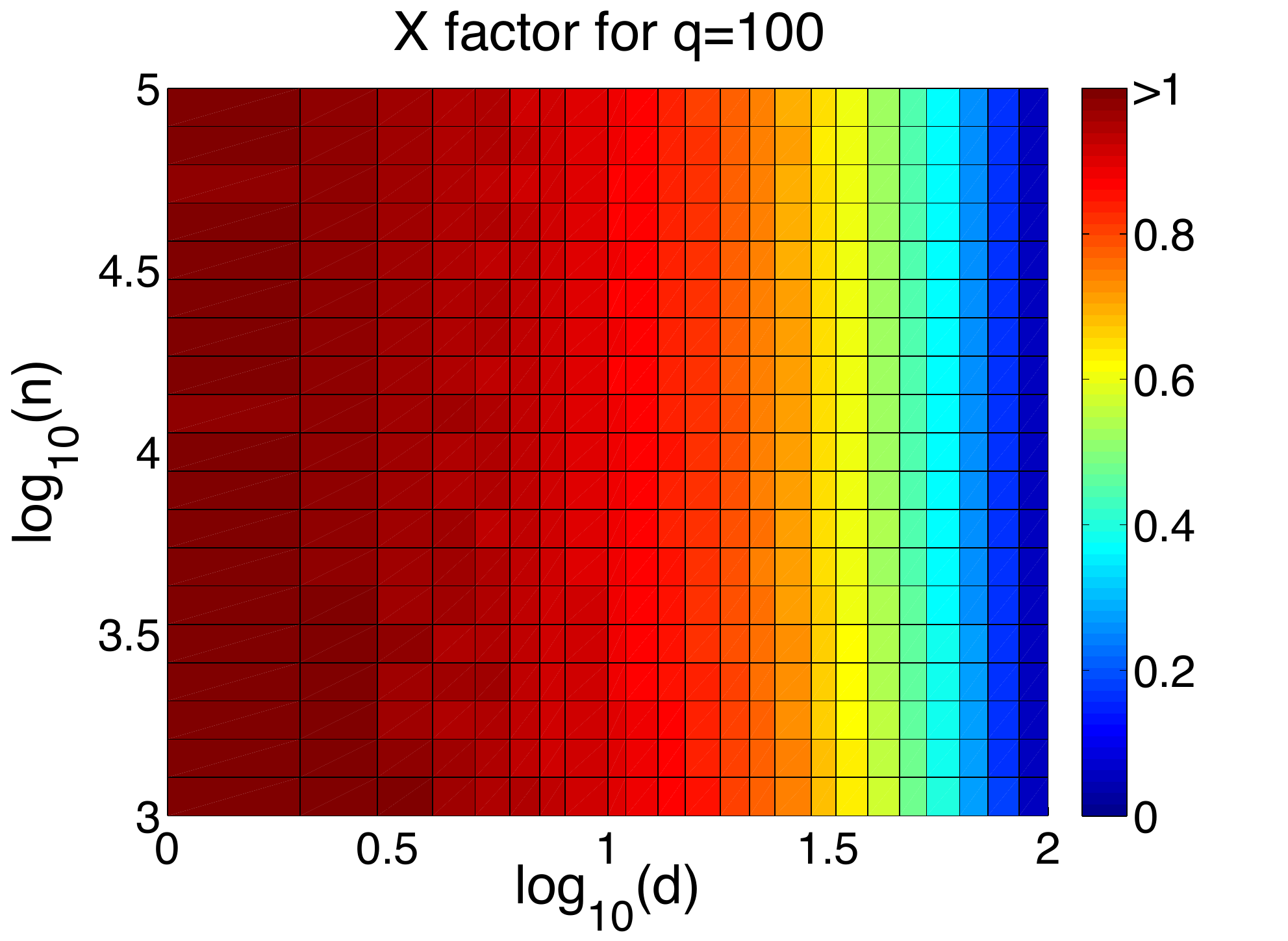}  \\
(c) & (d) 
\end{tabular}
\caption{Observed convergence factor $X$ for various values of $d$,
  $n$, and $q$.
$X$ is computed using \eqnok{eq:X}, using $N=500$ iterations of
  Algorithm~\ref{grouse:partial} per trial. (a) $d=10$, showing $X$
  averaged over ten trials for each $n$ and $q$. (b) $d=200$, with $X$
  plotted over one trial for each $n$ and $q$. (c) $n=10000$, with $X$
  plotted for $d$ and $q$, averaged over ten trials.  (d) $q=100$, with
  $X$ plotted for $n$ and $d$, averaged over ten trials. In (a) and (b), the darkest red takes value $1$; 
in (c) and (d), the darkest red is from $1$ to $1.4$. The plots all
  clearly show a phase transition in $X$, occurring when $q$ is some
  modest multiple of $d$, but with $X$ otherwise independent of $q$,
  $n$, and $d$, as suggested by the analysis. \label{fig:phase}}
\end{figure}

\section{Conclusion}

We have analyzed the GROUSE algorithm to find that near a solution,
GROUSE decreases subspace error at a linear rate, in expectation. Our
estimate of the linear rate depends on the problem dimensions and the
number of entries observed per vector, and matches well with our
computational observations.

We believe that there are deep connections between our analysis and
recent important work on randomized linear algebra (see, for example,
\cite{halko2011finding}). Often this work seeks to find a low-rank
approximation to a matrix or an approximation to its leading
eigenspace, and randomized column or row sampling is used to make
algorithms more efficient on large matrices. Our work randomly samples
matrix entries instead. A deeper understanding of how these approaches
are related is an interesting area for future investigation. GROUSE's
connections to the ISVD algorithm, which are explored further
in~\cite{grouseisvd}, may be a step in this direction.

Several other questions concerning GROUSE's convergence behavior
warrant further investigation.  We have observed empirically that
convergence to a solution occurs from any random starting point, given
a sufficiently large number of observed elements. This motivates us to
pursue a better mathematical understanding of the global convergence
properties.  Moreover, we observe convergence even for coherent
subspaces, and would like to understand why. Also of interest is the
behavior of GROUSE in the case of noisy observations. A diminishing
step size is needed here, leading to slower convergence. We would like
a better analytical understanding of this case under different noise
models of interest.

\section*{Acknowledgments} 

We are grateful to two referees for helpful and constructive comments
on the original version of this manuscript.

\appendix

\section{Proof of Theorem~\ref{th:sampledsingvals}} \label{app:sampledsingvals}

We start with a key result on matrix concentration.

\begin{theorem}[Noncommutative Bernstein Inequality~\cite{Gross09b, RechtImprovedMC09}]
\label{bernstein}
Let $X_1, \dots, X_m$ be independent zero-mean square $d \times d$
random matrices. Suppose 
\[
\rho_k^2 := \max \, \{\|\mathbb{E}[X_k X_k^T]\|_2,
\| \mathbb{E}[X_k^T X_k] \|_2 \}
\]
and $\|X_k\|_2 \leq M$ almost surely for all $k$. Then for any $\tau >
0$,
\[
\mathbb{P} \left[ \left\| \sum_{k=1}^m X_k \right\|_2 > \tau \right]
\leq 2d \exp \left( \frac{-\tau^2 / 2}{\sum_{k=1}^m \rho_k^2 + M\tau /
  3} \right).
\]
\end{theorem}

We proceed with the proof of Theorem~\ref{th:sampledsingvals}.
\begin{proof} 
We start by defining the notation
\[
u_k := U_{\Omega(k) \cdot}^T \in \R^d,
\] 
that is, $u_k$ is the transpose of the row of the row of $U$ that
corresponds to the $k$th element of $\Omega$. We thus define
\[
X_k := u_k u_k^T - \frac{1}{n}  I_d, 
\]
where $I_d$ is the $d \times d$ identity matrix. Because of orthonormality of
the columns of $U$, this random variable has zero mean.

To apply Theorem~\ref{bernstein}, we must compute the values of
$\rho_k$ and $M$ that correspond to this definition of $X_k$.  Since
$\Omega(k)$ is chosen uniformly with replacement, the $X_k$ are
distributed identically for all $k$, and $\rho_k$ is independent of
$k$ (and can thus be denoted by $\rho$).

Using the fact that 
\beq \label{eq:ABpsd}
\|A-B\|_2 \leq \max\{\|A\|_2,\|B\|_2\} \;\; \mbox{for
positive semidefinite matrices $A$ and $B$},
\eeq
and recalling that $\|u_k\|^2_2 = \|U_{\Omega(k) \cdot}\|^2_2 \leq
d\mu(U)/n$, we have
\[
\left\| u_k u_k^T - \frac{1}{n} I_d \right\|_2 \leq \max\left\{
\frac{d\mu(U)}{n}, \frac{1}{n} \right\}.
\]
Thus we can define $M := d\mu(U) / n$.
For $\rho$, we note by symmetry of $X_k$ that 
\begin{align}
\rho^2 =  \left\| \mathbb{E} \left[ X_k^2 \right] \right\|_2 
& =  \left\| \mathbb{E} \left[ u_k u_k^T u_k u_k^T - \frac{2}{n} u_k u_k^T + \frac{1}{n^2} I_d \right] \right\|_2 \nonumber \\
& =  \left\| \mathbb{E}\left[u_k (u_k^T u_k) u_k^T\right] -  \frac{1}{n^2} I_d  \right\|_2, \label{midtermreduce.alt}
\end{align}
where the last step follows from linearity of expectation,
and $E(u_k u_k^T) = (1/n) I_d$.

For the next step, we define $S$ to be the $n \times n$ diagonal
matrix with diagonal elements $\|U_{i \cdot}\|_2^2$,
$i=1,2,\dotsc,n$. We thus have
\[
\|E[u_k (u_k^T u_k) u_k^T]\|_2 = 
\left\| \frac{1}{n} U^T S U \right\| \le
\frac{1}{n} \| U \|_2^2 \|S\|_2 = \frac{1}{n} \frac{d \mu(U)}{n} =
\frac{d \mu(U)}{n^2}.
\]
Using \eqnok{eq:ABpsd}, we have from \eqnok{midtermreduce.alt} that
\[
\rho^2  \le 
\max \left( \left\| \mathbb{E}\left[u_k (u_k^T u_k) u_k^T\right] \right\|,
\frac{1}{n^2} \right) 
\le \max \left( \frac{d \mu(U)}{n^2},
\frac{1}{n^2} \right)  =
\frac{d \mu(U)}{n^2},
\]
since $d \mu(U) \ge d \ge 1$.

We now apply Theorem~\ref{bernstein}. First, we restrict $\tau$ to be
such that $M\tau \leq |\Omega| \rho^2$ to simplify the denominator of
the exponent. We obtain
\[
2d \exp \left( \frac{-\tau^2 / 2}{|\Omega| \rho^2 + M\tau/3}\right) \leq 2d\exp \left( \frac{-\tau^2 / 2}{\frac{4}{3} |\Omega| \frac{d\mu(U)}{ n^2}}\right),
\]
and thus
\[
\mathbb{P} \left[ \left\| \sum_{k \in \Omega} \left(u_k u_k^T -
  \frac{1}{n} I_d \right) \right\| > \tau \right] \leq 2d \exp \left(
\frac{-3 n^2 \tau^2}{8 |\Omega|d\mu(U)}\right).
\]
Now take $\tau = \gamma |\Omega|/n$ with $\gamma$ defined in the
statement of the lemma. Since $\gamma<1$ by assumption, $M\tau \leq
|\Omega| \rho^2$ holds and we have
\beq \label{eq:bni}
\mathbb{P} \left[ \left\| \sum_{k \in \Omega} \left(u_k u_k^T -
  \frac{1}{n} I_d \right) \right\|_2 \leq \frac{|\Omega|}{n} \gamma
  \right] \geq 1 - \delta.
\eeq
We have, by symmetry of $\sum_{k \in \Omega} u_k u_k^T$ and the fact
that 
\[
\lambda_i \left( \sum_{k \in \Omega} u_k u_k^T 
- \frac{|\Omega|}{n} I\right) = \lambda_i \left( \sum_{k \in \Omega} u_k
u_k^T \right) - \frac{|\Omega|}{n},
\]
that
\begin{align*}
\left\| \sum_{k \in \Omega} \left(u_k u_k^T -
  \frac{1}{n} I_d \right) \right\|_2  & =
\left\| \left( \sum_{k \in \Omega} u_k u_k^T \right) -
  \frac{|\Omega|}{n} I_d  \right\|_2 \\
& = 
\max_{i=1,2,\dotsc,n}  \left| \lambda_i \left( \sum_{k \in \Omega} u_k u_k^T  \right)  - \frac{|\Omega|}{n} \right|,
\end{align*}
  From \eqnok{eq:bni}, we have with probability $1-\delta$ that
\[
\lambda_i \left( \sum_{k \in \Omega} u_k
u_k^T \right)  \in \left[ (1-\gamma) \frac{|\Omega|}{n}, 
(1+\gamma) \frac{|\Omega|}{n} \right] \quad \mbox{for all $i=1,2,\dotsc,n$},
\]
completing the proof.
\end{proof}

\section{Proof of Lemma~\ref{lem:A}} \label{app:lemA}

 We drop the subscript ``$t$'' throughout the proof and use $A_+$ in
 place of $A_{t+1}$. From \eqref{eq:gupdate}, and using the
 definitions \eqref{eq:defs}, we have
\begin{align*}
A_+^T &= \bar{U}^TU_+ \\
&= \bar{U}^TU + \left\{
(\cos (\sigma \eta) - 1) \frac{\bar{U}^T UU^T \bar{U} s}{\|w\|} +
\sin (\sigma \eta) \frac{(I-\bar{U}^T UU^T \bar{U})s}{\|r\|} 
\right\} \frac{s^T\bar{U}^TU}{\|w\|} \\
&= \left\{
I + (\cos (\sigma \eta) - 1)  \frac{A^TAss^T}{\|w\|^2} +
\sin (\sigma \eta) \frac{(I-A^TA)ss^T}{\|r\| \|w\|}
\right\} A^T = HA^T,
\end{align*}
where the matrix $H$ is defined in an obvious way.
Thus
\[
\|A_+\|_F^2 = \trace (A_+A_+^T) = \trace (AH^THA^T).
\]
Focusing initially on $H^TH$ we obtain
\begin{align*}
H^TH &= I + (\cos (\sigma \eta)-1)^2 \frac{ss^T A^TAA^TA ss^T}{\|w\|^4} \\
& \quad +
(\cos (\sigma \eta)-1) \frac{ss^TA^TA + A^TAss^T}{\|w\|^2} \\
& \quad + \sin(\sigma \eta) \frac{2ss^T - ss^TA^TA - A^TAss^T}{\|r\|\|w\|} \\
& \quad +
2 \sin(\sigma \eta) (\cos(\sigma \eta)-1) \frac{ss^T A^TA ss^T - ss^T A^TAA^TAss^T}{\|r\| \|w\|^3} \\
& \quad + \sin^2 (\sigma \eta) \frac{s(s^Ts - 2s^TA^TAs+s^TA^TAA^TAs)s^T}{\|r\|^2 \|w\|^2}.
\end{align*}
It follows immediately that
\begin{align*}
A_+ A_+^T &= AA^T + (\cos (\sigma \eta)-1)^2 \frac{Ass^T A^TAA^TA ss^TA^T}{\|w\|^4} \\
& \quad +
(\cos (\sigma \eta)-1) \frac{Ass^TA^TAA^T + AA^TAss^TA^T}{\|w\|^2} \\
& \quad + \sin(\sigma \eta) \frac{2Ass^TA^T - Ass^TA^TAA^T - AA^TAss^TA^T}{\|r\|\|w\|} \\
& \quad +
2 \sin(\sigma \eta) (\cos(\sigma \eta)-1) \frac{Ass^T A^TA ss^TA^T - Ass^T A^TAA^TAss^TA^T}{\|r\| \|w\|^3} \\
& \quad + \sin^2 (\sigma \eta) \frac{As(s^Ts - 2s^TA^TAs+s^TA^TAA^TAs)s^TA^T}{\|r\|^2 \|w\|^2}.
\end{align*}
We now use repeatedly the fact that $\trace \, ab^T= a^Tb$ to deduce that
\begin{align*}
\trace (A_+A_+^T) &= \trace (AA^T) + (\cos (\sigma \eta)-1)^2 \frac{(s^TA^TAs)s^T A^TAA^TA s}{\|w\|^4} \\
& \quad +
(\cos (\sigma \eta)-1) \frac{2s^TA^TAA^TAs}{\|w\|^2} \\
& \quad + \sin(\sigma \eta) \frac{2s^TA^TAs - 2 s^TA^TAA^TAs}{\|r\|\|w\|} \\
& \quad +
2 \sin(\sigma \eta) (\cos(\sigma \eta)-1) \frac{(s^T A^TA s)^2 - (s^T A^TAA^TAs)(s^TA^TAs)}{\|r\| \|w\|^3} \\
& \quad + \sin^2 (\sigma \eta) \frac{\|s\|^2 s^TA^TAs -  2(s^TA^TAs)^2 + (s^TA^TAA^TAs)(s^TA^TAs)}{\|r\|^2 \|w\|^2}.
\end{align*}
Now using $w=As$ (and hence $s^TA^TAs=\|w\|^2$), we have
\begin{align*}
\trace (A_+A_+^T) &= \trace (AA^T) + (\cos (\sigma \eta)-1)^2 \frac{s^T A^TAA^TA s}{\|w\|^2} \\
& \quad +
(\cos (\sigma \eta)-1) \frac{2s^TA^TAA^TAs}{\|w\|^2} \\
& \quad + 2 \sin(\sigma \eta) \frac{\|w\|^2 - s^TA^TAA^TAs}{\|r\|\|w\|} \\
& \quad +
2 \sin(\sigma \eta) (\cos(\sigma \eta)-1) \frac{\|w\|^2 - s^T A^TAA^TAs}{\|r\| \|w\|} \\
& \quad + \sin^2 (\sigma \eta) \frac{\|s\|^2-  2\|w\|^2 + (s^TA^TAA^TAs)}{\|r\|^2},
\end{align*}
For the second and third terms on the right-hand side, we use the
identity
\[
(\cos (\sigma \eta)-1)^2 + 2(\cos (\sigma \eta)-1) = \cos^2 (\sigma
\eta)-1 = -\sin^2 (\eta \sigma),
\]
allowing us to combine these terms with the final $\sin^2(\sigma
\eta)$ term. Using also the identity $\|r\|^2 = \|s\|^2-\|w\|^2$, we
obtain for the combination of these three terms that 
\begin{align*}
& \sin^2 (\sigma \eta) \left[ 1 - \frac{\|w\|^2}{\|r\|^2} + s^TA^TAA^TAs
\left( \frac{1}{\|r\|^2} - \frac{1}{\|w\|^2} \right) \right] \\
 & \qquad = \sin^2 (\sigma \eta) \left( 1 - \frac{\|w\|^2}{\|r\|^2}\right)
\left(1- \frac{s^TA^TAA^TAs}{\|w\|^2} \right).
\end{align*}
We can also combine the third and fourth terms in the right-hand side
above to yield a combined quantity
\[
2 \sin(\sigma \eta) \cos(\sigma \eta) \frac{\|w\|}{\|r\|} 
\left(1- \frac{s^TA^TAA^TAs}{\|w\|^2} \right).
\]
By substituting these two compressed terms into the expression above,
we obtain
\begin{align*}
& \trace (A_+A_+^T)  = \trace (AA^T)  \\
& \quad + \sin(\sigma \eta)
\left(1- \frac{s^TA^TAA^TAs}{\|w\|^2} \right) \left[
\left(1-\frac{\|w\|^2}{\|r\|^2} \right) \sin (\sigma \eta) + 2 \cos
(\sigma \eta) \frac{\|w\|}{\|r\|} \right].
\end{align*}
We now use the relations \eqnok{eq:rtheta} to deduce that
\[
\frac{\|w\|}{\|r\|} = \frac{\cos \theta}{\sin \theta}, \quad
1-\frac{\|w\|^2}{\|r\|^2} = -\frac{\cos (2 \theta)}{\sin^2 \theta},
\]
and thus the increment $\trace (A_+A_+^T)  - \trace (AA^T)$ becomes
\begin{align*}
& \sin (\sigma \eta) 
\left(1- \frac{s^TA^TAA^TAs}{\|w\|^2} \right) 
\left[ -\frac{\cos (2 \theta)}{\sin^2 \theta} \sin (\sigma \eta) + 2 \cos (\sigma \eta) \frac{\cos \theta}{\sin \theta} \right] \\
& \quad =  \frac{\sin (\sigma \eta)  \sin(2 \theta - \sigma \eta)}{\sin^2 \theta}
\left(1- \frac{s^TA^TAA^TAs}{\|w\|^2} \right).
\end{align*}
The result \eqref{eq:Aincr} follows by substituting $w=As$ and
\eqref{eq:sigma}.

Nonnegativity of the right-hand side follows from $\theta_t \ge 0$,
$2\theta_t - \sigma_t \eta_t \ge 0$, and $\| A_t^T w_t \| \le \| 
\bar{U}^T U_t \| \| w_t \| \le \| w_t \|$.

To prove that the right-hand side of \eqref{eq:Aincr} is zero when
$v_t \in \cS$ or $v_t \perp \cS$, we take the former case first.
Here, there exists $\hat{s}_t \in \R^d$ such that
\[
v_t = \bar{U} s_t = U_t \hat{s}_t.
\]
Thus
\[
w_t = A_t s_t = U_t^T \bar{U} s_t = U_t^T U_t \hat{s}_t = \hat{s}_t, 
\]
so that $\| v_t \| = \|w_t \|$ and thus $\theta_t=0$, from
\eqnok{eq:rtheta}. This implies that the right-hand side of
\eqref{eq:Aincr} is zero. When $v_t \perp \cS_t$, we have $w_t =
U_t^Tv_t=0$ and so $\theta_t = \pi/2$ and $\sigma_t=0$, implying again
that the right-hand side of \eqref{eq:Aincr} is zero.

\bibliographystyle{siam} 
\bibliography{grouserefs}

\newpage

\end{document}